\documentclass[12pt]{amsart}

\usepackage[utf8]{inputenc}
\usepackage{setspace}
\usepackage[margin=1in]{geometry}
\usepackage{graphicx}
\graphicspath{ {./figures/} }
\usepackage{subcaption}
\usepackage{amsmath}
\usepackage{lineno}
\usepackage{tikz}
\usepackage{cancel}
\usepackage{amsthm}
\usepackage{amsfonts}
\usepackage{amssymb}\usepackage{comment}
\usepackage{hyperref}
\usepackage{tikz-cd}
\usepackage{enumerate, enumitem}

\usepackage{nicematrix}

\usepackage[normalem]{ulem}

\theoremstyle{definition}
\newtheorem{theorem}{Theorem}[section]
\newtheorem{lemma}[theorem]{Lemma}
\newtheorem{conjecture}[theorem]{Conjecture}
\newtheorem{proposition}[theorem]{Proposition}

\newtheorem{definition}[theorem]{Definition}

\newtheorem{example}[theorem]{Example}

\newtheorem{remark}[theorem]{Remark}
\newtheorem{notation}[theorem]{Notation}

\definecolor{Red}{rgb}{1,0,0}
\definecolor{Blue}{rgb}{0,0,1}
\definecolor{Green}{rgb}{0,1,0}
\definecolor{black}{rgb}{0.0, 0.0, 0.0}

\newcommand{\In}{\textit{In}}
\newcommand{\Out}{\textit{Out}}
\newcommand{\Leak}{\textit{Leak}}
\newcommand{\ds}{\displaystyle}
\newcommand{\msc}{\mathsf{C}}

\newcommand{\R}{\mathbb{R}}

\begin{document}

\title[Identifiability of mammillary models]{Parameter identifiability of \\ linear-compartmental mammillary models}
\author[Clemens]{Katherine Clemens}
\address{Bryn Mawr College, Bryn Mawr PA, USA. {\em Present address:} Rochester Institute of Technology, Rochester NY, USA}

\author[Martinez]{Jonathan Martinez}
\address{University of Southern California, Los Angeles CA, USA}

\author[Shiu]{Anne Shiu}
\address{Texas A\&M University, College Station TX, USA}

\author[Thompson]{Michaela Thompson}
\address{Fairfield University, Fairfield CT, USA}

\author[Warren]{Benjamin Warren}
\address{Texas A\&M University, College Station TX, USA}

\date{November 17, 2025}


\begin{abstract}
Linear compartmental models are a widely used tool for analyzing systems arising in biology, medicine, and more.  In such settings, it is essential to know whether model parameters can be recovered from experimental data.  This is the identifiability problem. 
For a class of linear compartmental models with one input and one output, namely, those for which the underlying graph is a bidirected tree, 
Bortner {\em et al.}\
completely characterized which such models are structurally identifiability, which means that every parameter is generically locally identifiable.  
Here, we delve deeper, by examining which individual parameters are locally versus globally identifiable. Specifically, we analyze mammillary models, which consist of one central compartment which is connected to all other (peripheral) compartments.  
For these models, which fall into five infinite families, we determine which individual parameters are locally versus globally identifiable, and we give formulas for some of the globally identifiable parameters in terms of the coefficients of input-output equations.  Our proofs rely on a combinatorial formula due to Bortner {\em et al.}\ for these coefficients.

\vspace{.1in}

\vspace{.1in}
\end{abstract}

\maketitle

\section{Introduction}
 Compartmental models are used frequently in applications ranging from ecology~\cite{R.J.} to disease modeling~\cite{review-compartment-disease} and pharmacology~\cite{Cherruault_1992}.  
 A well-known example arises in pharmacokinetics, in which 
 linear compartmental models 
 are widely used to predict the concentration of an injected drug in the body over time; such knowledge is essential for determining optimal drug dosage and timing.  
 Linear compartmental models additionally have been used to describe the distribution of potassium between plasma and red blood cells, the distribution and turnover of cholesterol in the human body, 
 and the flow of energy within ecosystems~\cite{jacquez}.
 
 
 In such applications, it is crucial to determine whether the parameters of  models can be uniquely determined from data~\cite{structural-covid, animal, natural-resource-modeling, wieland2021structural}.  This is the problem of parameter {\em identifiability}.  Our focus here is on {\em structural identifiability}, which assumes the ideal situation of noiseless data.  
Access to noiseless data is generally not realistic, but assessing structural identifiability is nevertheless of great importance, as it is a prerequisite for {\em practical identifiability}, which allows for data with noise.

The inspiration for our work is a recent classification of (generic) local identifiability for a family of linear compartmental models, namely, those
in which the underlying graph is a bidirected tree.  This classification is due to Bortner {\em et al.}~\cite{BGMSS}, and it states the following: {\em A bidirected-tree model with one input and one output is generically locally identifiable if and only if the model has at most $1$ leak, and the distance from the input to output is at most~$1$}.  In other words, this result tells us exactly which models have \uline{all} parameters (at least) generically locally identifiable.  Such models are called {\em identifiable}.  

Bortner {\em et al.}'s result on bidirected-tree models motivates the following questions:
\begin{enumerate}[label=({\bf Q\arabic*})]
    \item In identifiable bidirected-tree models, which parameters are generically \uline{globally} identifiable (that is, uniquely identifiable)?
        \item In unidentifiable bidirected-tree models, what is the source of this unidentifiability, that is, which parameters are unidentifiable?        
    \end{enumerate}
The significance of question~({\bf Q1})  in applications comes from the fact that being able to estimate parameters uniquely is preferable to being able to do so only up to a finite set.  Similarly, regarding question~({\bf Q2}), being able to estimate some but not all parameters may suffice in some applications, so we wish to know which ones can or can not be recovered.

Our work addresses questions ({\bf Q1}) and ({\bf Q2}) for an 
 infinite family of bidirected-tree models.  These models are {\em mammillary models}, which means that the underlying graph is a star graph, consisting of a central compartment  connected to a number of peripheral compartments.  Mammillary models are of ``considerable practical importance''~\cite[pg.\ 31]{godfrey}, and instances of mammillary models in biological applications appear in~\cite{Cherruault_1992, godfrey, jacquez, environment, pharmacokinetic}.


\begin{figure}
    \centering
    \includegraphics[width=6in]{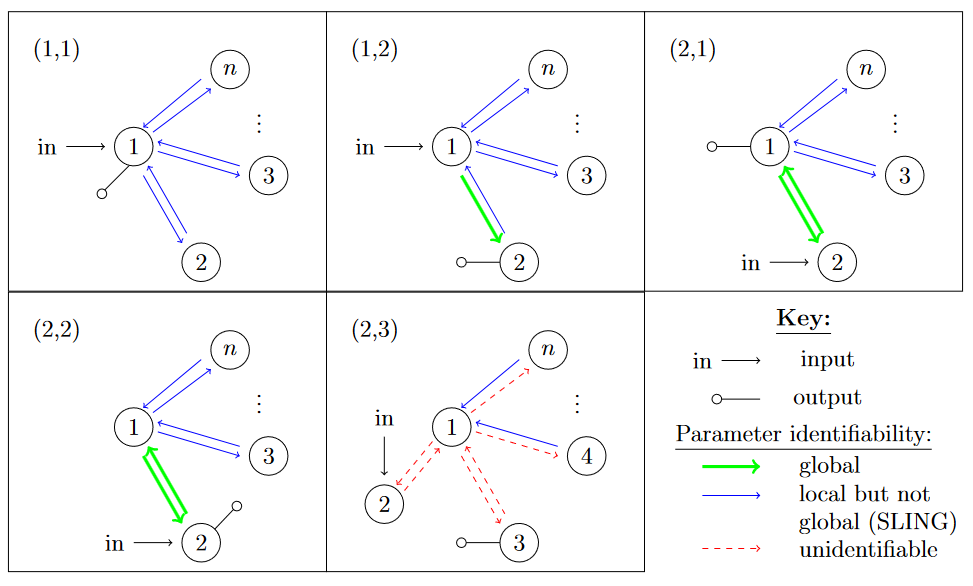}
    \caption{{\bf Summary of main results}.   
    All mammillary models (up to symmetry) with one input, one output, and no leaks are depicted.  The model with input in compartment-$i$ and output in compartment-$j$ is labeled by $(i,j)$. Identifiability properties of all parameters are shown; generically {globally} identifiable parameters are indicated by (green) boldface arrows, generically {locally} (but not globally) identifiable parameters are shown in (blue) plain arrows, and unidentifiable parameters are shown with (red) dashed arrows.  These properties are proven or conjectured in this work (see Theorem \ref{thm:summary} and Conjecture~\ref{conj:2-3}). The result for the model (1,1) was already known~\cite[\S 4.1]{cobelli-lepschy-romaninjacur}.    }
    \label{fig:summary}
\end{figure}

Five (families of) mammillary models are depicted in Figure~\ref{fig:summary}, which summarizes our main results.  These results concern all mammillary models with one input, one output, and no leaks.  Specifically, we prove the identifiability properties -- generically globally identifiable versus generically locally identifiable versus unidentifiable -- of most of the edge parameters, as shown in Figure~\ref{fig:summary} (Theorem~\ref{thm:summary}).  There are, however, several edges whose identifiability properties we verified computationally for several values of $n$ (here, $n$ denotes the number of compartments), but currently can not prove for all $n$.  
These edges are the edges in model~$(2,3)$ that are marked as unidentifiable (indicated by dashed arrows) in Figure~\ref{fig:summary} (Conjecture~\ref{conj:2-3}).  

 In summary, for the models in Figure~\ref{fig:summary}, we 
 completely answer question ({\bf Q1}) and we conjecture an answer to question ({\bf Q2}).
We emphasize that, to our knowledge, we are essentially the first to investigate questions~({\bf Q1}) and~({\bf Q2}) for infinite families of models.  
(Chau investigated similar questions for reparametrized versions of mammillary models with leaks at all compartments~\cite{chau-mam} and so-called catenary models~\cite{chau-cat}; while Cobelli, Lepschy, and Romanin Jacur investigated the mammillary models we label by (1,1) and other models~\cite{cobelli-lepschy-romaninjacur} -- see also~\cite[\S 5.8]{cobelli2002identifiability} -- but their focus was slightly different.)
On the other hand, for 
models arising in biological applications (cellular signaling, epidemiology, physiology, and other areas), 
Barreiro and Villaverde recently investigated the occurrence (and origins of) parameters that are generically locally -- but not globally -- identifiable (so-called SLING parameters, short for ``structurally locally identifiable but not globally'')~\cite{sling}. 

Our proofs, like those in several prior works~\cite{cycle-cat, bortner2024, 
BGMSS,
BM-2022,
BM-indisting,
CJSSS, GOS, GHMS, singularlocus}, 
involve analyzing input-output equations, which are equations involving parameters, input variables, output variables, and their derivatives.  This technique comes from following the differential-algebra approach to identifiability.  We also rely on a recent combinatorial formula for coefficients of input-output equations~\cite{BGMSS}, which allows us to analyze whether (and, when possible, how) these coefficients can be used to recover some or all of the parameters.

This article is structured as follows. Section~\ref{sec:background} provides background on linear compartmental models and their identifiability. Our results are proven in Section~\ref{sec:results}, and we conclude with a discussion in Section~\ref{sec:discussion}.

\section{Background} \label{sec:background}
In this section, we introduce linear compartmental models (Section~\ref{sec:linear-model}) and their
input-output equations (Section~\ref{sec:i-o-equations}).
Subsequently, we define  
identifiability of models (Section~\ref{sec:identifiability}) and of individual parameters (Section~\ref{sec:identifiability-parameters}).
Our notation matches that of prior work~\cite{GOS,MSE}.

\subsection{Linear compartmental models} \label{sec:linear-model}
Informally, a linear compartmental model is a directed graph, with certain vertices (compartments) marked as ``inputs'' (representing inflows into the system), ``outputs'' (compartments at which experimental measurements can be taken), and ``leaks'' (outflows from the system). The formal definition follows.

\begin{definition} \label{def:compartmental-model}
A \textbf{linear compartmental model}, denoted by $(G, \In, \Out, \Leak)$, consists of a directed graph $G = (V, E)$, in which each vertex $i\in V$ represents a \textbf{compartment} of the model, together with three sets $\In, \Out, \Leak \subseteq V$ which denote the sets of {\bf input}, {\bf output}, and {\bf leak} compartments, respectively. 
\end{definition}

In a linear compartmental model $(G, \In, \Out, \Leak)$, a directed edge $j\to i$ in $G$ represents a flow from compartment-$j$ to compartment-$i$. Each such edge has an associated parameter~$k_{ij}$. Similarly, each leak compartment $p \in \Leak$ has an associated parameter, $k_{0 p}$.  However, in this work, we focus on models without leaks (that is, $\Leak = \varnothing$).

In figures, we represent a linear compartmental model by its graph $G$, together with 
arrows for the inputs and the leaks, plus the following symbol for outputs: \begin{tikzpicture}
    \draw (0,0) circle (0.06);
    \draw (0.06,0.06) -- (.16,.19);
\end{tikzpicture}.  See, for instance, Figures~\ref{fig:summary} and~\ref{fig:mammillary}.  The models depicted in both figures belong to an important class of linear compartmental models, called 
``mammillary models'', which is the focus of this work. 

\begin{definition} \label{def:types-of-models} 
Let $\mathcal M = (G, \In, \Out, \Leak)$ be a linear compartmental model.
\begin{enumerate}
    \item $\mathcal M$ is a {\bf mammillary model} if $G=(V,E)$ is a bidirected-star graph: $V=\{1,2,\dots, n\}$, where $n \geq 3$, and $E=\{ 1 \leftrightarrows 2, 
    ~ 1 \leftrightarrows 3,~ \dots,~ 1 \leftrightarrows n
    \}$.
    \item $\mathcal{M}$ is \textbf{strongly connected} if $G$ is strongly connected, that is, for each unordered pair of vertices $i,j$ in $G$, there is a directed path in $G$ from $i$ to $j$ and also a directed path from $j$ to~$i$.
\end{enumerate}
\end{definition}
\noindent
Every mammillary model is strongly connected.

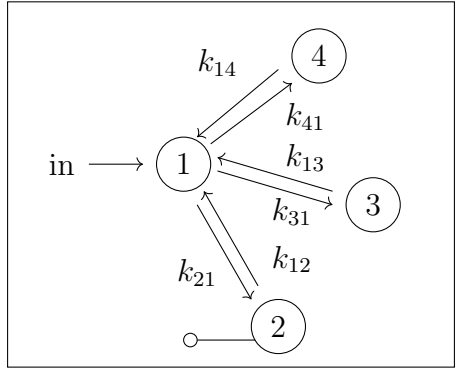
\begin{figure}[ht]
\begin{center}
	\begin{tikzpicture}[scale=1.8]
\draw (3,0) circle (0.2);
 	\draw (3.7,-1.2) circle (0.2);
 	\draw (4.4,-0.3) circle (0.2);
 	\draw (4 ,.8) circle (0.2);
    	\node[] at (3, 0) {1};
    	\node[] at (3.7, -1.2) {2};
    	\node[] at (4.4, -0.3) {3};
    	\node[] at (4, .8) {$4$};
     \draw[->] (3.1, -.3) -- (3.5, -1);
	 \draw[<-] (3.15, -.2) -- (3.55, -0.9);
	 \draw[->] (3.25, -0.05) -- (4.1, -0.3);
	 \draw[<-] (3.25, 0.05) -- (4.1, -0.2);
	 \draw[->] (3.2, .15) -- (3.8, .6);
	 \draw[<-] (3.1, 0.2) -- (3.7, .7);
   	 \node[] at (3.1, -.8) {$k_{21}$};
   	 \node[] at (3.8, -0.7) {$k_{12}$};
   	 \node[] at (3.9, .05) {$k_{13}$};
   	 \node[] at (3.8, -0.35) {$k_{31}$};
   	 \node[] at (3.25, 0.75) {$k_{14}$};
   	 \node[] at (3.9, 0.35) {$k_{41}$};
 	\draw (3.05,-1.3) circle (0.05);	
	 \draw[-] (3.53, -1.3 ) -- (3.1, -1.3);	
  
	 \draw[->] (2.3, 0) -- (2.7, 0);	
   	 \node[] at (2.1, 0) {in};

\draw (1.7,-1.5) rectangle (5., 1.2);
	\end{tikzpicture}
\end{center}
\caption{The mammillary model $\mathcal M = (G, \In, \Out, \Leak)$ with $4$ compartments and $\In=\{1\}$, $ \Out=\{2\}$, and $\Leak= \varnothing $.}
\label{fig:mammillary}
\end{figure}

\begin{definition} \label{def:compartmental-matrix}
Let
$\mathcal{M}=(G, \In, \Out, \Leak)$ be a linear compartmental model
with $n$ compartments, where $G=(V,E)$. 
The \textbf{compartmental matrix} of $\mathcal{M}$ is the $n{\times}n$ matrix $A = (a_{ij})$ with entries as follows:
    \begin{equation*}
        a_{ij} ~:=~
            \begin{cases}
                \ds -k_{0i} - \sum_{\{ p \mid (i,p) \in E \} } k_{pi}  &\text{if $i=j$ and $i \in \Leak$},\\
                \ds - \sum_{ \{ p \mid (i,p) \in E \}} k_{pi} &\text{if $i=j$ and $i \notin \Leak$},\\
                k_{ij} &\text{if $i\neq j$ and $j \to i$ is an edge of $G$},\\
                0 &\text{if $i\neq j$ and $j\to i$ is not an edge of $G$}.
            \end{cases}.
    \end{equation*}
\end{definition}

\begin{example} \label{ex:mammillary-compartmental-matrix}
The mammillary model in Figure~\ref{fig:mammillary} has the following compartmental matrix:
\begin{align} \label{eq:compartmental-matrix-running-example}
    A
    ~=~\left[\begin{array}{ccccc}
    -k_{21}-k_{31}-k_{41} & k_{12} & k_{13} & k_{14} \\
    k_{21} & -k_{12} & 0  & 0\\
    k_{31} & 0 & -k_{13} & 0\\
    k_{41} & 0 & 0 &-k_{14}
    \end{array}\right]~.
\end{align}
%
\end{example}

Next, we describe how to use the compartmental matrix to write the ordinary differential equation (ODE) system arising from a  model $\mathcal M = (G, \In, \Out, \Leak)$. 
Let $n$ denote the number of compartments, and 
let 
$x(t) = (x_1(t), x_2(t), \hdots, x_n(t))^{T}$ denote the vector of concentrations of all compartments at time $t$.  For $i \in \In$, let 
$u_{i}(t)$ denote the 
inflow rate into compartment-$i$ at time $t$.
For $i \in \Out$, let $y_i(t)$ denote the corresponding output variable (measurement) at time $t$.  
Now 
$\mathcal M$ defines the following ODE system with inputs, where $A$ is the compartmental matrix of $\mathcal M$:
\begin{align} \label{eq:ode}
    x'(t) &~=~ Ax(t) + u(t)~, \\
    y_i(t) &~=~ x_i(t)~, \ \text{ for all } i \in \Out~,
    \notag
\end{align}
where $u_i(t):=0$ for all $i \not\in \In$. 

\begin{example}[Example~\ref{ex:mammillary-compartmental-matrix}, continued] \label{ex:mammilary-ODE}
Using the 
compartmental matrix in~\eqref{eq:compartmental-matrix-running-example}, the ODE system~\eqref{eq:ode} for the model in Figure~\ref{fig:mammillary} is as follows:
\begin{align} 
\label{eq:odeEx} \notag
    x_1'(t) &~=~ (-k_{21}-k_{31}-k_{41}) x_1 + k_{12} x_2 + k_{13} x_3 + k_{14} x_4 + u_1(t) \\ \notag
    x_2'(t) &~=~ k_{21} x_1 -k_{12} x_2 \\ \notag
    x_3'(t) &~=~ k_{31} x_1 -k_{13} x_3 \\ \notag
    x_4'(t) &~=~ k_{41} x_1 -k_{14} x_4 \\ \notag
    y_2(t) &~=~ x_2(t)~. 
    \notag
\end{align}
\end{example}

\subsection{Input-output equations} \label{sec:i-o-equations}
An {\bf input-output equation} is an equation that holds along every solution $x(t)$ of the ODEs~\eqref{eq:ode} and involves only the parameters $k_{ij}$, input variables~$u_{\ell}$, output variables $y_m$, and their derivatives. 
The next result is a formula for input-output equations, which is due to Meshkat, Sullivant, and Eisenberg~\cite[Theorem 2]{MSE}. 

\begin{notation} \label{notation:submatrix}
    We let $B^{j,i}$ denote the submatrix obtained from a matrix $B$ by removing row-$j$ and column-$i$.  
\end{notation}

\begin{proposition}[Input-output equations~\cite{MSE}] \label{prop:in-out-equations} 
Let $\mathcal{M} = (G, \In, \Out, \Leak)$ be a linear compartmental model with $n$ compartments and at least one input. Let $A$ be the compartmental matrix. For $i\in \Out$, the following equation is an input-output equation for $\mathcal{M}$:
\begin{equation} \label{eq:in-out}
    \det(\partial I-A)y_{i} ~=~ \sum_{j\in In}(-1)^{i+j}\det\left[(\partial I-A)^{j,i}\right]u_j~,
\end{equation}
\noindent 
where $\partial I$ denotes the $n{\times}n$ matrix in which each diagonal entry is the differential operator $d/dt$ and all off-diagonal entries are 0.
\end{proposition}

\begin{example}[Example~\ref{ex:mammilary-ODE}, continued] \label{ex:mamillary-i-o-eqn}
Recall that the model in Figure~\ref{fig:mammillary} has $\In=\{1\}$ and $\Out=\{2\}$.  We use the compartmental matrix in~\eqref{eq:compartmental-matrix-running-example} to obtain the input-output equation~\eqref{eq:in-out}:
\begin{multline*}
    \det \left(
    \begin{bmatrix}
    d/dt+
    k_{21}+k_{31}+k_{41} & -k_{12} & -k_{13} & 
    -k_{14} \\
    -k_{21} & d/dt +k_{12} & 0  & 0\\
    -k_{31} & 0 & d/dt +k_{13} & 0\\
    -k_{41} & 0 & 0 &d/dt+k_{14}
    \end{bmatrix} \right) 
    y_{2} 
    \\
    ~=~
    (-1)^{2+1}\det
    \left(
    \begin{bmatrix}
    -k_{21} &  0  & 0\\
    -k_{31} &  d/dt +k_{13} & 0\\
    -k_{41} &  0 &d/dt+k_{14}
    \end{bmatrix} \right) 
    u_1~,
\end{multline*}
which expands as follows: 
\begin{align} \label{eq:i-o-equation-running-example}
    \notag
        y_{2}^{(4)} 
        & + (k_{12}+k_{13} + k_{14} + k_{21} + k_{31} + k_{41})y_{2}^{(3)} 
        \\
    \notag
         &  + (k_{12}k_{13} + k_{12}k_{14} + k_{12}k_{31} + k_{12}k_{41} + k_{13}k_{14} + k_{13}k_{21} + k_{13}k_{41} + 
    k_{14}k_{21} + k_{14}k_{31})y_{2}^{(2)} 
        \\
    \notag
        &  
    + (k_{12}k_{13}k_{14} + k_{12}k_{13}k_{41} + k_{12}k_{14}k_{31} + k_{13}k_{14}k_{21})y_{2}^{(1)} 
    \\
    & \quad \quad \quad ~=~ 
    k_{21} u_{1}^{(2)} + (k_{13}k_{21} + k_{14}k_{21})u_{1}^{(1)} + k_{13}k_{14}k_{21} u_{1}
\end{align}
\end{example}

Next, we introduce notation for an index set for all parameters of a model $\mathcal{M} =  (G, \In, \Out, \Leak)$, where $G=(V,E)$:
    \begin{align} \label{eq:index-set-param}
        \mathcal{P}_{\mathcal{M}}
        ~:=~
        \{
        (j,i) \mid (j,i) \in E~,~ \mathrm{or}~i=0~\mathrm{and}~ j\in\Leak
        \}~.
            \end{align}
We use this set, 
which has size $\lvert E\rvert + \lvert \Leak\rvert$,
together with the input-output equation(s)~\eqref{eq:in-out}, to define the 
\textbf{coefficient map} of $\mathcal M$:
$$\msc : \mathbb{R}^{ \lvert E\rvert + \lvert \Leak\rvert} \rightarrow \mathbb{R} ^m~,$$ 
which
evaluates each vector of parameters $(k_{ij})_{(j,i) \in \mathcal{P}_{\mathcal{M}}}$ at 
the vector of non-constant coefficients of the input-output equation(s).
(Here, $m$ denotes the number of such coefficients.)

\begin{example}[Example~\ref{ex:mamillary-i-o-eqn}, continued] \label{ex:mammillary-coefficient-map}
For the model in Figure~\ref{fig:mammillary}, we use the input-output equation~\eqref{eq:i-o-equation-running-example} to obtain the coefficient map $\msc : \mathbb{R}^6 \to \mathbb{R}^6 $, where 
\begin{align} 
    \notag
    \msc_1 ~:&=~
    k_{12}+k_{13} + k_{14} + k_{21} + k_{31} + k_{41} \\ \notag
    \msc_2 ~:&=~ k_{12}k_{13} + k_{12}k_{14} + k_{12}k_{31} + k_{12}k_{41} + k_{13}k_{14} + k_{13}k_{21} + k_{13}k_{41} + 
    k_{14}k_{21} + k_{14}k_{31}\\ 
    \label{eq:coeff-running-example}
    \msc_3 ~:&=~ k_{12}k_{13}k_{14} + k_{12}k_{13}k_{41} + k_{12}k_{14}k_{31} + k_{13}k_{14}k_{21}\\ \notag
    \msc_4 ~:&=~ k_{21}\\ \notag
    \msc_5 ~:&=~ k_{13}k_{21} + k_{14}k_{21}\\ \notag
    \msc_6 ~:&=~ k_{13}k_{14}k_{21}~.
\end{align}
\end{example}

The next result, Proposition~\ref{prop:coeff-formula}
 below, is used to analyze coefficient maps.  We first need the following definition.

 \begin{definition}
     \label{def:spanning-incoming-forest}
     Let $H$ be a directed graph.
     \begin{enumerate}
         \item      A {\bf spanning subgraph} is 
     a subgraph of $H$ that includes all vertices of $H$.  
    \item     
    A \textbf{spanning incoming forest}
     of $H$ 
     is a spanning subgraph such that:
        \begin{itemize}
            \item      each node has at most one outgoing edge, and 
            \item     the underlying undirected multigraph
     is a forest, that is, has no cycles.
        \end{itemize}
     \end{enumerate}
 \end{definition}

Bortner {\em et al.}\ gave a combinatorial formula, which is in terms of spanning incoming forests, for the coefficients of input-output equations~\cite[Theorem~3.1]{BGMSS}.  We state the version of their result for the case of models with one input, one output, and no leaks (which is the situation we focus on in this work), as follows.

\begin{proposition}[Coefficients of input-output equations for linear compartmental models~\cite{BGMSS}] 
\label{prop:coeff-formula}
    Let $n \geq 3$.
    Let $\mathcal{M} = (G, \In, \Out, \Leak)$
    be an $n$-compartment model with one input, one output, and no leaks: $\In=\{j\}$, $\Out = \{i\}$, and $\Leak= \varnothing$.  Write the input-output equation~\eqref{eq:in-out} as follows:
 	\begin{align}
    \label{eq:eq:i-o-c-and-d}
	 y_i^{(n)} + c_{n-1}  y_i^{(n-1)} + \dots  + c_1 y_i' + c_0 y_i ~=~ 
		d_{n-1} u_j^{(n-1)} + 
			 \dots  + d_{1} u_j' + d_{0} u_j
                ~.
	\end{align}
	Then the coefficients of the input-output equation~\eqref{eq:eq:i-o-c-and-d} are given by:
    \begin{align} \label{eq:bortner-formula-coeff}
    c_k ~&=~ \sum_{F \in \mathcal{F}_{n-k}(G)} \pi_F \quad \text{ for } k=0,1, \ldots , n-1~, \ \text{ and } \\
    \notag
    d_k ~&=~ \sum_{F \in \mathcal{F}^{j,i}_{n-k-1}(G^*_i)} \pi_F \quad \text{ for } k=0,1,\ldots, n-1~,
    \end{align}
where:  
\begin{itemize}
    \item $G^*_i$ is the directed graph obtained from $G$ by removing all outgoing edges from vertex-$i$ (the output), 
	\item  $\mathcal{F}_{\ell} ( G )$ is the set of all spanning incoming  forests of $G$ with exactly $\ell$ edges,  
	\item  $ \mathcal{F}_{\ell}^{j, i } ( G_i^* )$ is the set of all spanning incoming forests of $G_i^*$ with exactly $\ell$ edges, such that some connected component (of the underlying undirected graph) contains both of the vertices $j$ and $i$, 
    \item 
    $\pi_F$ is the product of edge labels of a graph $F$, that is,
    $\pi_F := \prod_{ e \in E_F } L(e)~$, where 
    $L(e)$ is the label of edge $e$, and 
    $E_F$ is set of edges of $F$.  If $E_F= \varnothing$, then $\pi_F:=1$.
\end{itemize} 
\end{proposition}

\begin{remark}[$c_0=d_{n-1}=0$] \label{rem:c0=0}
    In the context of Proposition~\ref{prop:coeff-formula}, 
    we have $c_0=0$.  To see this, observe that the formula for $c_0$ in~\eqref{eq:bortner-formula-coeff} is a sum over 
    certain $n$-edge cycle-free subgraphs of $G$.  However, $G$ has only 
    $n$ vertices, and hence has \uline{no}
    $n$-edge cycle-free subgraphs.  We conclude that $c_0$ is the empty sum and hence is $0$.
    An alternate proof can be given directly from~\eqref{eq:in-out}.  Similarly, $d_{n-1}=0$.
\end{remark}

\begin{remark} \label{rem:error-in-bortner-et-al}
In Proposition~\ref{prop:coeff-formula}, the input-output equation~\eqref{eq:eq:i-o-c-and-d} corrects a sign error in~\cite[Theorem~3.1]{BGMSS}, which erroneously contains a factor of $
             (-1)^{i+j} $ on the right-hand side.
\end{remark}

\begin{example}[Example~\ref{ex:mammillary-coefficient-map}, continued] \label{ex:mammillary-borten-map}
Recall that the $4$-compartment mammillary model 
$\mathcal M = (G, \In, \Out, \Leak)$
in Figure~\ref{fig:mammillary} has 
$\In=\{1\}$, $ \Out=\{2\}$, and $\Leak= \varnothing $.  Following the notation in~\eqref{eq:eq:i-o-c-and-d}, the coefficients $c_i$ and $d_i$ are given by $(c_3, c_2, c_1, d_2, d_1, d_0):= (\msc_1, \msc_2, \msc_3, \msc_4, \msc_5, \msc_6)$, where the $\msc_i$ are as in~\eqref{eq:coeff-running-example}, and also $c_0=0$. Next, the graphs $G$ and $G_2^*$ are shown in Figures~\ref{fig:original-graph-G} and~\ref{fig:G1}, respectively.
It is now straightforward to check that the coefficients $c_i$ and $d_i$ satisfy the formulas~\eqref{eq:bortner-formula-coeff}.  For instance, $c_1=\msc_3$ is a sum over the $4$ spanning incoming forests of $G$ with $4-1=3$ edges; these $4$ forests are depicted in Figure  ~\ref{fig:spanning-inc-forest-G}.  
Similarly, $d_1= \msc_5$ is a sum over the $2$ spanning incoming forests of $G_2^*$ with $4-1-1=2$ edges in which compartments $1$ and $2$ are connected; these $2$ forests are shown in Figure ~\ref{fig:spanning-inc-forest-G*}.  
\end{example}

\begin{center}
\begin{figure}
\centering
\begin{minipage}{.4\textwidth}
\centering
	\begin{tikzpicture}[scale=1.7]
\draw (3,0) circle (0.2);
 	\draw (3.7,-1.2) circle (0.2);
 	\draw (4.4,-0.3) circle (0.2);
 	\draw (4 ,.8) circle (0.2);
    	\node[] at (3, 0) {1};
    	\node[] at (3.7, -1.2) {2};
    	\node[] at (4.4, -0.3) {3};
    	\node[] at (4, .8) {$4$};
    	\node[] at (2.6, -1.3) {$G$};
     \draw[->] (3.1, -.3) -- (3.5, -1);
	 \draw[<-] (3.15, -.2) -- (3.55, -0.9);
	 \draw[->] (3.25, -0.05) -- (4.1, -0.3);
	 \draw[<-] (3.25, 0.05) -- (4.1, -0.2);
	 \draw[->] (3.2, .15) -- (3.8, .6);
	 \draw[<-] (3.1, 0.2) -- (3.7, .7);
   	 \node[] at (3.1, -.8) {$k_{21}$};
   	 \node[] at (3.7, -0.7) {$k_{12}$};
   	 \node[] at (3.9, .05) {$k_{13}$};
   	 \node[] at (3.8, -0.35) {$k_{31}$};
   	 \node[] at (3.25, 0.75) {$k_{14}$};
   	 \node[] at (3.9, 0.35) {$k_{41}$};
\draw (2.3,-1.5) rectangle (5., 1.2);
	\end{tikzpicture}

    \captionof{figure}{The graph $G$ for the mammillary model in Figure~\ref{fig:mammillary}.}
    \label{fig:original-graph-G}
\end{minipage}%
\begin{minipage}{.7\textwidth}
\centering
	\begin{tikzpicture}[scale=1.7]
\draw (3,0) circle (0.2);
 	\draw (3.7,-1.2) circle (0.2);
 	\draw (4.4,-0.3) circle (0.2);
 	\draw (4 ,.8) circle (0.2);
    	\node[] at (3, 0) {1};
    	\node[] at (3.7, -1.2) {2};
    	\node[] at (4.4, -0.3) {3};
    	\node[] at (4, .8) {$4$};
    	\node[] at (2.6, -1.3) {$G_2^*$};
     \draw[->] (3.1, -.3) -- (3.5, -1);
	 \draw[->] (3.25, -0.05) -- (4.1, -0.3);
	 \draw[<-] (3.25, 0.05) -- (4.1, -0.2);
	 \draw[->] (3.2, .15) -- (3.8, .6);
	 \draw[<-] (3.1, 0.2) -- (3.7, .7);
   	 \node[] at (3.1, -.8) {$k_{21}$};
   	 \node[] at (3.9, .05) {$k_{13}$};
   	 \node[] at (3.8, -0.35) {$k_{31}$};
   	 \node[] at (3.25, 0.75) {$k_{14}$};
   	 \node[] at (3.9, 0.35) {$k_{41}$};
\draw (2.3,-1.5) rectangle (5., 1.2);
	\end{tikzpicture}
    \caption{The graph $G_{2}^*$, obtained from the graph $G$ in Figure~\ref{fig:original-graph-G} by removing the outgoing edge from compartment-2 (the output in Figure~\ref{fig:mammillary}).}
    \label{fig:G1}
    \end{minipage}

    \begin{minipage}{1\textwidth}
    
\centering
\vspace{0.5cm}
	\begin{tikzpicture}[scale=1.4]
 \begin{scope}[xshift=-4.5cm] 
\draw (3,0) circle (0.2);
 	\draw (3.7,-1.2) circle (0.2);
 	\draw (4.4,-0.3) circle (0.2);
 	\draw (4 ,.8) circle (0.2);
    	\node[] at (3, 0) {1};
    	\node[] at (3.7, -1.2) {2};
    	\node[] at (4.4, -0.3) {3};
    	\node[] at (4, .8) {$4$};
	 \draw[<-] (3.15, -.2) -- (3.55, -0.9);
	 \draw[<-] (3.25, 0.05) -- (4.1, -0.2);
	 \draw[<-] (3.1, 0.2) -- (3.7, .7);
   	 \node[] at (3.7, -0.7) {$k_{12}$};
   	 \node[] at (3.9, .05) {$k_{13}$};
   	 \node[] at (3.25, 0.75) {$k_{14}$};
\draw (2.3,-1.5) rectangle (5., 1.2);
 \end{scope}

 \begin{scope}[xshift=-1.5cm] 
\draw (3,0) circle (0.2);
 	\draw (3.7,-1.2) circle (0.2);
 	\draw (4.4,-0.3) circle (0.2);
 	\draw (4 ,.8) circle (0.2);
    	\node[] at (3, 0) {1};
    	\node[] at (3.7, -1.2) {2};
    	\node[] at (4.4, -0.3) {3};
    	\node[] at (4, .8) {$4$};
     \draw[->] (3.1, -.3) -- (3.5, -1);
	 \draw[<-] (3.25, 0.05) -- (4.1, -0.2);
	 \draw[<-] (3.1, 0.2) -- (3.7, .7);
   	 \node[] at (3.1, -.8) {$k_{21}$};
   	 \node[] at (3.9, .05) {$k_{13}$};
   	 \node[] at (3.25, 0.75) {$k_{14}$};
\draw (2.3,-1.5) rectangle (5., 1.2);
 \end{scope}

 \begin{scope}[xshift=1.5cm] 
\draw (3,0) circle (0.2);
 	\draw (3.7,-1.2) circle (0.2);
 	\draw (4.4,-0.3) circle (0.2);
 	\draw (4 ,.8) circle (0.2);
    	\node[] at (3, 0) {1};
    	\node[] at (3.7, -1.2) {2};
    	\node[] at (4.4, -0.3) {3};
    	\node[] at (4, .8) {$4$};
	 \draw[<-] (3.15, -.2) -- (3.55, -0.9);
	 \draw[->] (3.25, -0.05) -- (4.1, -0.3);
	 \draw[<-] (3.1, 0.2) -- (3.7, .7);
   	 \node[] at (3.7, -0.7) {$k_{12}$};
   	 \node[] at (3.8, -0.35) {$k_{31}$};
   	 \node[] at (3.25, 0.75) {$k_{14}$};
\draw (2.3,-1.5) rectangle (5., 1.2);
 \end{scope}

 \begin{scope}[xshift=4.5cm] 
\draw (3,0) circle (0.2);
 	\draw (3.7,-1.2) circle (0.2);
 	\draw (4.4,-0.3) circle (0.2);
 	\draw (4 ,.8) circle (0.2);
    	\node[] at (3, 0) {1};
    	\node[] at (3.7, -1.2) {2};
    	\node[] at (4.4, -0.3) {3};
    	\node[] at (4, .8) {$4$};
	 \draw[<-] (3.15, -.2) -- (3.55, -0.9);
	 \draw[<-] (3.25, 0.05) -- (4.1, -0.2);
	 \draw[->] (3.2, .15) -- (3.8, .6);
   	 \node[] at (3.7, -0.7) {$k_{12}$};
   	 \node[] at (3.9, .05) {$k_{13}$};
   	 \node[] at (3.9, 0.35) {$k_{41}$};
\draw (2.3,-1.5) rectangle (5., 1.2);
 \end{scope}
	\end{tikzpicture}

    \captionof{figure}{Spanning incoming forests of $G$ with 3 edges.}
    \label{fig:spanning-inc-forest-G}
\end{minipage}%

\begin{minipage}{1\textwidth}
    
\centering
\vspace{0.5cm}
	\begin{tikzpicture}[scale=1.4]
 \begin{scope}[xshift=-1.5cm] 
\draw (3,0) circle (0.2);
 	\draw (3.7,-1.2) circle (0.2);
 	\draw (4.4,-0.3) circle (0.2);
 	\draw (4 ,.8) circle (0.2);
    	\node[] at (3, 0) {1};
    	\node[] at (3.7, -1.2) {2};
    	\node[] at (4.4, -0.3) {3};
    	\node[] at (4, .8) {$4$};
     \draw[->] (3.1, -.3) -- (3.5, -1);
	 \draw[<-] (3.25, 0.05) -- (4.1, -0.2);
   	 \node[] at (3.1, -.8) {$k_{21}$};
   	 \node[] at (3.9, .05) {$k_{13}$};
\draw (2.3,-1.5) rectangle (5., 1.2);
 \end{scope}

 \begin{scope}[xshift=1.5cm] 
\draw (3,0) circle (0.2);
 	\draw (3.7,-1.2) circle (0.2);
 	\draw (4.4,-0.3) circle (0.2);
 	\draw (4 ,.8) circle (0.2);
    	\node[] at (3, 0) {1};
    	\node[] at (3.7, -1.2) {2};
    	\node[] at (4.4, -0.3) {3};
    	\node[] at (4, .8) {$4$};
     \draw[->] (3.1, -.3) -- (3.5, -1);
	 \draw[<-] (3.1, 0.2) -- (3.7, .7);
   	 \node[] at (3.1, -.8) {$k_{21}$};
   	 \node[] at (3.25, 0.75) {$k_{14}$};
\draw (2.3,-1.5) rectangle (5., 1.2);
 \end{scope}
	\end{tikzpicture}
    \captionof{figure}{Spanning incoming forests of $G_2^*$ with 2 edges.}
    \label{fig:spanning-inc-forest-G*}
\end{minipage}%

\end{figure}
\end{center}

\subsection{Identifiability of models} \label{sec:identifiability}

The idea behind the problem of model identifiability is as follows:
Given measured values of how much enters and exits the input and output compartments, respectively, can we infer all parameter values of a model?  To be more precise, we are asking whether, from generic values of the inputs and initial conditions, we can recover the parameters (at least locally, which means up to a finite set) from exact measurements of the inputs and the outputs. This is the concept of {\em generic local identifiability}. 

This notion of identifiability is equivalent (for strongly connected models with at least one input) to  Definition~\ref{def:ident} below, which checks identifiability by way of coefficient maps. 
This equivalence was proven by Ovchinnikov, Pogudin, and Thompson~\cite[Main Result 3]{OPT}.  

In what follows, we use ``almost every'' in the usual measure-theoretic sense: 
``{\bf almost every} $x$ in a set $X$'' means ``for all
$x$ in $X \smallsetminus Z$, for some measure-zero subset $Z$ of $X$''.

\begin{definition}[Identifiability of $\mathcal{M}$]\label{def:ident} 
Let $ \mathcal M =(G, \In, \Out, \Leak)$
be
a linear compartmental model that is strongly connected and has at least one input.  
Let $G=(V,E)$, and let $\msc :  \R^{|E| + |\Leak|} \rightarrow \R ^m$ denote the coefficient map of $\mathcal M$.
\begin{enumerate}
    \item  $\mathcal M$ is \textbf{generically locally identifiable} 
    if
     almost every
     point in $\R^{|E| + |\Leak|}$ has an open neighborhood on which $\msc$ is injective.
    \item $\mathcal M$ is {\bf unidentifiable} if $\mathcal M$ is \uline{not} generically locally identifiable.
\end{enumerate}
\end{definition}

%
%

The following result, which was proved by Bortner~{\em et~al.}~\cite[Corollary~5.4]{BGMSS}, characterizes identifiability of mammillary models (Definition~\ref{def:types-of-models}(1)) with one input and one output.

\begin{proposition}[Identifiability of mammillary models] \label{prop:iden-mammillary}
    A mammillary model 
    $ \mathcal M =(G, \In, \Out, \Leak)$
    with $|\In|=|\Out|=1$  is generically locally identifiable
    if and only if 
    $|Leak| \leq 1$ 
    and one or more of the following hold:
    (1) $In=Out$, (2) $In=\{1\}$, or (3) $Out=\{1\}$.
\end{proposition}

\begin{example}[Example~\ref{ex:mammillary-borten-map}, continued]
    \label{ex:running-ex-iden-of-model}
    For the mammillary model in Figure~\ref{fig:mammillary},
    we have 
     $In = \{1\}$, $\Out=\{2\}$,
     and no leaks.  
    Therefore, 
    by Proposition~\ref{prop:iden-mammillary}, the model is generically locally identifiable. 
    In the next subsection, we examine whether the individual parameters of this model are generically globally identifiable versus generically locally identifiable.
\end{example}

\subsection{Identifiability of parameters} \label{sec:identifiability-parameters}

Definition~\ref{def:ident} above concerns the identifiability of models, whereas the next definition pertains to the identifiability of individual parameters of a model (cf.~\cite[Definition 2.5]{HOPY}).

\begin{definition}[Identifiability of parameter $k_{pq}$] \label{def:iden-param}
Let $ \mathcal M =(G, \In, \Out, \Leak)$ be a strongly connected, linear compartmental model with at least one input.  Let $G=(V,E)$, and let $\msc :  \R^{|E| + |\Leak|} \rightarrow \R ^m$ denote the coefficient map of $\mathcal M$.
Let $k_{pq}$ be a parameter of $\mathcal M$, i.e., $(q,p)$ is in the index set $\mathcal{P}_{\mathcal{M}}$, in~\eqref{eq:index-set-param}.  
Write each parameter vector 
$(k_{i j})_{(j,i) \in \mathcal{P}_{\mathcal{M}}}$ 
as 
$( k_{pq};~ \widetilde k )$, 
where 
$\widetilde k := (k_{ij})_{(j,i) \in \mathcal{P}_{\mathcal{M}} \smallsetminus \{ (q,p)\}}$; 
this allows us to write the coefficient map as $\msc = \msc( k_{pq};~ \widetilde k)$.

The parameter $k_{pq}$ is: 
    \begin{enumerate}
        \item {\bf globally identifiable} if, 
        for every $ ( k^*_{pq};~ \widetilde k^* ) $ in $ \mathbb{R}^{|E| + |\Leak|}$,
        the following set 
        (which is a projection of a preimage of $\msc( k^*_{pq};~ \widetilde k^*)$)
        has size one:
        \begin{align} \label{eq:set-parameters}
            \{
            k_{pq} \mid
            \mathrm{there~exists~}
            \widetilde{k} \in 
            \mathbb{R}^{|E| + |\Leak|}
            \mathrm{~such~that~}
            \msc( k_{pq};~ \widetilde k) =
            \msc( k^*_{pq};~ \widetilde k^*)
            \}~;
           \end{align}
        \item {\bf generically globally identifiable} if, 
        for almost every
        $ ( k^*_{pq};~ \widetilde k^* ) $ in $ \mathbb{R}^{|E| + |\Leak|}$, the set~\eqref{eq:set-parameters} has size one;
        \item {\bf generically locally identifiable}
        if, 
        for almost every
        $ ( k^*_{pq};~ \widetilde k^* ) $ in $ \mathbb{R}^{|E| + |\Leak|}$, 
        the set~\eqref{eq:set-parameters} is finite; and
\item {\bf unidentifiable} if $k_{pq}$ is \uline{not} generically locally identifiable.
    \end{enumerate}
\end{definition}

\begin{notation} \label{notation:sling}
Following 
    Barreiro and Villaverde, 
we use {\bf SLING} (which stands for ``structurally locally identifiable but not globally'') to refer to a parameter that is generically locally identifiable, but {\it not} generically globally identifiable~\cite{sling}.
\end{notation}

\begin{remark}[Identifiability of models versus parameters] \label{rem:gli-param}
    A model is generically locally identifiable 
    (Definition~\ref{def:ident}(1))
    if and only if all of its parameters are generically locally identifiable (Definition~\ref{def:iden-param}(3)); cf.~\cite[Remark~2.6(a)]{HOPY}.  
    Equivalently, a model is unidentifiable if and only if at least one of its parameters is unidentifiable.
\end{remark}

\begin{remark}[Software for checking identifiability] \label{rem:sian}
Identifiability of parameters can be checked using software.  One option is
the {\tt Maple} software {\tt SIAN} (Structural Identifiability ANalyser)~\cite{SIAN}, 
which is based on theory developed by 
Hong, Ovchinnikov, Pogudin, and Yap~\cite{HOPY}. 
(This theory is developed over $\mathbb{C}$, whereas we work over $\mathbb{R}$; hence, users of {\tt SIAN} need to be careful when interpreting the output.) 
Another option is the {\tt Julia} package {\tt StructuralIdentifiability.jl}, which implements theory of 
Dong, Goodbrake, Harrington, and Pogudin~\cite{julia-software}.  
\end{remark}

\begin{remark}[Restricting parameter values]
    A meaningful extension of Definition~\ref{def:iden-param} would be to
    include, as part of the setup, a subset of the full parameter space $\mathbb{R}^{|E| + |Leak|}$ (or subsets of $\mathbb{R}$, one for each parameter).  The reason would be to allow a researcher to take into account known information about parameter values (positivity, for instance, or restriction to an interval).  Here, however, for simplicity, we do not incorporate this addition.
\end{remark}

\subsubsection{Identifiable functions} 
One of our primary methods for proving that certain parameters are generically globally identifiable or generically locally identifiable is through the theory of ``identifiable functions'' (see Definition~\ref{def:identifiable-function}
and Proposition~\ref{prop:polynom-to-iden} below).
This approach makes precise the idea that identifiable parameters are the ones that we can ``solve for'' in terms of the coordinates of the coefficient map.

We follow the notation of Meshkat, Rosen, and Sullivant~\cite{Meshkat-Rosen-Sullivant} in what follows.

\begin{definition}
    \label{def:identifiable-function}
    Let $\msc:\mathbb{R}^{\lvert E\rvert + \lvert \Leak\rvert} \rightarrow \mathbb{R}^m$ be the coefficient map of a linear compartmental model, and let $f:\mathbb{R}^{\lvert E\rvert + \lvert \Leak\rvert} \rightarrow \mathbb{R}$ be some other function. The function $f$ is:
\begin{itemize}
    \item an \textbf{identifiable function} from $\msc$ if, for all $p,p'\in \mathbb{R}^{\lvert E\rvert + \lvert \Leak\rvert}$, we have that $\msc(p)=\msc(p')$ implies that $f(p)=f(p')$.

    \item a \textbf{generically identifiable function} from $\msc$ if $f$ is identifiable from $\msc$ on some open, dense subset $U\subseteq \mathbb{R}^{\lvert E\rvert + \lvert \Leak\rvert}$.


    \item a \textbf{locally identifiable function} from $\msc$ if there is an open, dense subset $U\subseteq \mathbb{R}^{\lvert E\rvert + \lvert \Leak\rvert}$ such that, for all $p\in U$, there is an open neighborhood $U_p$ of $p$ such that $f$ is identifiable from $\msc$ on $U_p$.
\end{itemize} 

\end{definition}

\begin{remark}[Identifiable parameters and functions]
    \label{rem:iden-fnc-vs-param}
    The relationship between Definitions~\ref{def:iden-param} and~\ref{def:identifiable-function} is as follows.
    For a parameter $k_{ij}$ 
    of a model with coefficient map $\msc$, if the function $f=k_{ij}$ is identifiable (respectively, generically identifiable or locally identifiable), then the parameter $k_{ij}$ is globally identifiable (respectively, generically globally identifiable or generically locally identifiable).
    \end{remark}
The following proposition, which is similar in spirit to~\cite[Proposition~4.4]{Meshkat-Rosen-Sullivant}
(and can be viewed as an instantiation of the ideas in~\cite[Proposition~3.4]{HOPY} and~\cite[Definition~2.5]{multi-experiment}),
will help us prove identifiability properties of parameters:

\begin{proposition}[] \label{prop:polynom-to-iden}
Let $ \mathcal M =(G, \In, \Out, \Leak)$
be
a linear compartmental model with coefficient map $\msc :  \R^{|E| + |\Leak|} \rightarrow \R ^m$.
Let $k_{ij}$ be a parameter of $\mathcal{M}$.  
Suppose that there is a function 
$ g: 
\R \times \R^{|E| + |\Leak|} \rightarrow \R
$, which we write as $g = g(z;(k_{i j})_{(j,i) \in \mathcal{P}_{\mathcal{M}}})$, 
for which:
\begin{enumerate}
    \item for some $d\geq 1$, there exist polynomial
    functions $q_0,q_1,\dots, q_d: \R^{|E| + |\Leak|} \rightarrow \R$
    that are 
    identifiable from $\msc$, such that 
    \[
        g ~=~ q_d z^d + q_{d-1}z^{d-1} + \dots + q_0~,
    \]
    and the leading coefficient $q_d$ is \uline{not} the zero function;
    and 
    \item the function $g|_{z=k_{ij}}: \R^{|E| + |\Leak|} \rightarrow \R$ is the zero function.
    \end{enumerate}
Then the parameter $k_{ij}$ is generically \uline{locally} identifiable.  Moreover, if $d=1$, then the parameter $k_{ij}$ is generically \uline{globally} identifiable.
\end{proposition}

\begin{proof}
Let $g,~q_0,q_1,\dots, q_d$, and $k_{ij}$ be as in the statement of the proposition.  The set 
$U:=
\R^{|E| + |\Leak|} \smallsetminus \{q_d=0\}$ is an open, dense subset of $\R^{|E| + |\Leak|}$ 
(here we use the fact that the polynomial $q_d$ is nonzero).  In particular, the complement of ${U}$ has measure zero.  

    Fix $(k^*_{ij}; \widetilde{k}^*) \in {U}$ (here, the notation $(k^*_{ij}; \widetilde{k}^*)$ is like that in Definition~\ref{def:iden-param}).  
    By Definition~\ref{def:iden-param}(2), it suffices to show that there are only finitely many values $k^{**}_{ij}$ 
    for which there exist some $ \widetilde{k}^{**} \in \R^{|E| + |\Leak|-1}$ 
    with 
    $
    \msc(k^{**}_{ij}; \widetilde{k}^{**}) =
    \msc(k^*_{ij}; \widetilde{k}^*)$.  

    To this end, recall that $q_0,q_1,\dots, q_d$
     are identifiable from $\msc$.  This fact implies that
    the following holds for all $i=0,1,\dots,d$, whenever $
    \msc(k^{**}_{ij}; \widetilde{k}^{**}) =
    \msc(k^*_{ij}; \widetilde{k}^*)$:
    \begin{align*}
        q_i^* ~&:=~
            q_i|_{(k_{ij}; \widetilde{k})=(k^*_{ij}; \widetilde{k}^*)}
            ~=~
            q_i|_{(k_{ij}; \widetilde{k})=(k^{**}_{ij}; \widetilde{k}^{**)}}~.
    \end{align*}
    Now recall that $g|_{z=k_{ij}}$ is the zero function (by assumption).  It now follows that, if $
    \msc(k^{**}_{ij}; \widetilde{k}^{**}) =
    \msc(k^*_{ij}; \widetilde{k}^*)$, then $k^{**}_{ij}$ is a root of the following univariate polynomial (which is nonzero because we constructed 
   ${U}$ to be disjoint from $\{q_d = 0\}$):
   \[
   g|_{(k_{ij}; \widetilde{k})=(k^*_{ij}; \widetilde{k}^*)}
   ~=~
   q^*_d z^d + q^*_{d-1}z^{d-1} + \dots + q^*_0~.
   \]
    Thus, $k^{**}_{ij}$ can take at most $d$ values, and so we conclude that the parameter $k_{ij}$ is generically locally identifiable.  Additionally, if $d=1$, then there is only one possible value (namely, $k^{**}_{ij} = k^{*}_{ij}$) and so, in this case, the parameter $k_{ij}$ is generically globally identifiable.
    \end{proof}

\begin{example}[Example~\ref{ex:running-ex-iden-of-model}, continued] \label{ex:mammillary-solve}
We return to the model in Figure~\ref{fig:mammillary} and its coefficient map~\eqref{eq:coeff-running-example}.  We saw that $\msc_4= k_{21}$, and so the parameter $k_{21}$ is easily recovered from the coefficients $\msc_i$.  More precisely, $k_{21}$ is an identifiable function from $\msc$, so 
Remark~\ref{rem:iden-fnc-vs-param} implies that $k_{21}$ is \sout{generically} globally identifiable.  
Alternatively, we can draw 
a slightly weaker
conclusion by noting that $k_{21}$ satisfies the hypotheses of 
Proposition~\ref{prop:polynom-to-iden} with $g:=-k_{21}+z$ (so, $q_0=-k_{21}$ and $q_1=1$), and hence $k_{21}$ is generically globally identifiable.  
It turns out that no other parameter of this model is generically globally identifiable; in fact, all remaining parameters are SLING.  This fact can be proven directly or by examining the symmetry in the model (see Example~\ref{ex:symmetric-edge} below). 
%
%
%
\end{example}

In Example~\ref{ex:mammillary-solve}, we saw that one of the coefficients (namely, $\msc_4$) equals the parameter $k_{21}$.  This parameter corresponds to the edge from the input (at compartment-$1$) to the output (at compartment-$2$). This observation generalizes, as follows. 

\begin{lemma}[Global identifiability of edge from input to output] \label{lem:edge-in-to-out-global}
    Let  $ \mathcal M =(G, \In, \Out, \Leak)$ be a strongly connected, linear compartmental model with $\In=\{i\}$, $\Out=\{j\}$, and $\Leak = \varnothing$.  If $i \to j$ is an edge of $G$ (with parameter $k_{ji}$), then $k_{ji}$ is 
    globally identifiable.
\end{lemma}

\begin{proof}
Proposition~\ref{prop:coeff-formula} implies that one of the coefficients of the input-output equations, namely, $d_{n-2}$, as in~\eqref{eq:eq:i-o-c-and-d}, equals $k_{ji}$. Now the result follows from Remark~\ref{rem:iden-fnc-vs-param}.
\end{proof} 

\begin{remark}
Gogishvili constructed a large database of linear compartmental models, and observed that Lemma~\ref{lem:edge-in-to-out-global} holds for all models in the database~\cite[\S5]{gogishvili-database}. 
\end{remark}

\subsubsection{Symmetric edge parameters} \label{sec:iden-symmetry}
We end this section by asserting that edges that are symmetric share the same identifiability properties (Lemma~\ref{lem:symmetry} below). This simple idea has appeared in prior work (e.g.,~\cite[\S2.2]{gogishvili-database} and~\cite{cobelli-lepschy-romaninjacur}), and we motivate this result through the following example.

\begin{example}[Example~\ref{ex:mammillary-solve}, continued]
\label{ex:symmetric-edge}
    Observe in Figure~\ref{fig:mammillary} the symmetry between compartments $3$ and $4$.  This symmetry tells us that the parameters for edges incident to $3$ (namely, $k_{13}$ and $k_{31}$) and those incident to $4$ (namely, $k_{14}$ and $k_{41}$) are \uline{not} generically globally identifiable, because their values are indistinguishable from those of their symmetric pair.  
    This indistinguishability also can be seen in the coefficients~\eqref{eq:coeff-running-example} of the input-output equation: Each coefficient is invariant under exchanging $k_{13}$ and $k_{14}$
    with $k_{31}$ and $k_{41}$, respectively.  We make these ideas precise through the next definition and lemma.
\end{example}

\begin{definition} \label{def:automorphism}
Let $n$ be a positive integer.
    \begin{enumerate}
        \item An {\bf automorphism} of a directed graph $G=(V,E)$, where $V=\{1,2,\dots, n\}$, is a permutation $\sigma$ of $V$ such that $\sigma(G):= (V, \sigma(E)) = G$, where $\sigma(E):= \{\sigma(e) \mid e \in E \}$ and $\sigma(e):= (\sigma(i),\sigma(j))$, if $e=(i,j) \in E$.
        \item 
        For
        linear compartmental models $ \mathcal M =(G, \In, \Out, \Leak)$ and $ \mathcal M' =(G', \In', \Out', \Leak')$ with $G=G'$, 
        a \textbf{morphism} from $\mathcal M$ to $\mathcal M'$ is an automorphism $\sigma$ of $G$ such that $\sigma(\In) = \In'$, $\sigma(\Out) = \Out',$ and $\sigma(\Leak)=\Leak'$.
        \item An \textbf{automorphism} of a linear compartmental model $\mathcal M$ is a morphism from $\mathcal M$ to~$\mathcal M$.
    \end{enumerate}
\end{definition}

\begin{remark}[Input-output equations under morphisms] \label{rem:morphism-in-out}
Given a morphism from $\mathcal M$ to $\mathcal M'$, it is straightforward to see from definitions that the input-output equations~\eqref{eq:in-out} for $\mathcal M'$ are obtained from the input-output equations for $\mathcal M$ by permuting the subscripts according to $\sigma$ (i.e., we replace 
$k_{ij}$ and $k_{0j}$
by, respectively, $k_{\sigma(i),\sigma(j)}$ and $k_{0 \sigma(j)}$, for relevant indices $i,j$).
\end{remark}

\begin{lemma}[Symmetric edge parameters are SLING or unidentifiable] \label{lem:symmetry}
    Consider a
    linear compartmental model
    $ \mathcal M =(G, \In, \Out, \Leak)$, where $G=(V,E)$. 
    Consider two edges, $e=(q,p) \in E$ and $e'=(q',p') \in E$, with $e \neq e'$.  If there exists an automorphism $\sigma$ of $\mathcal{M}$ such that $\sigma(e)= e'$, then the parameters for $e$ and $e'$ (namely, $k_{pq}$ and $k_{p'q'}$) are
    either 
    both 
    SLING parameters 
    or 
    both unidentifiable parameters 
    of $\mathcal{M}$.
\end{lemma}

\begin{proof}
Let $\mathcal{M}$ be a linear compartmental model, with 
automorphism $\sigma$ 
and
edges $e=(q,p)$ and $e'=(q',p')$, as in the statement of the lemma.  
Our first aim is to show that the parameters $k_{pq}$ and $k_{p'q'}$ have the same identifiability properties (as in Definition~\ref{def:iden-param}).  Accordingly, 
consider any parameter vector
${\mathbf k}^* 
\in \mathbb{R}^{|E| + |\Leak|}$. 
It suffices to show that the projected set~\eqref{eq:set-parameters} for the edge $e=(q,p)$ is equal to the corresponding projected set for $e'=(q',p')$.  In fact, we need only show one containment, as the reverse containment will follow by applying the same argument to the inverse morphism $\sigma^{-1}$.

Let $\msc:\mathbb{R}^{|E|+|\Leak|} \to \mathbb{R}^m$ denote the coefficient map of $\mathcal{M}$, and let $\pi_{pq}: \mathbb{R}^{|E|+|\Leak|} \to \mathbb{R}$ (respectively, $\pi_{p'q'}$) denote the projection onto the coordinate $k_{pq}$ (respectively, $k_{p'q'}$).
Now consider the following diagram, where $\widetilde{\sigma}$ permutes the coordinates according to $\sigma$ (the coordinates  
$k_{ij}$ and $k_{0j}$ are sent to, respectively, $k_{\sigma(i),\sigma(j)}$ and $k_{0 \sigma(j)}$, for relevant indices $i,j$):

\begin{center}
    
\begin{tikzcd}
    \mathbb{R}^{|E|+|\Leak|} \arrow[r, "\pi_{pq}"] \arrow[d, "\msc"] \arrow[rd, "\widetilde \sigma"] & \mathbb{R}             \\
    \mathbb{R}^m                       & \mathbb{R}^{|E|+|\Leak|} \arrow[u, "\pi_{p'q'}"'] \arrow[l, "\msc"]
\end{tikzcd}

\end{center}

In this diagram, the upper triangle commutes by construction.  The lower triangle also commutes, by Remark~\ref{rem:morphism-in-out} and the fact that $\sigma$ is an automorphism.  

We claim that the fact that the diagram above commutes, implies that the projected set~\eqref{eq:set-parameters} for $e$, namely, 
$\pi_{pq}(\msc^{-1}(\msc({\mathbf k}^*)))$ 
is contained in the projected set for $e'$, namely, 
$\pi_{p'q'}(\msc^{-1}(\msc({\mathbf k}^*)))$. 
Indeed, this is a straightforward ``diagram chasing'' argument, which we outline as follows.  Given $x \in \pi_{pq}(\msc^{-1}(\msc({\mathbf k}^*)))$, let $\overline{y} := \widetilde{\sigma}(\overline{x})$, where $\pi_{pq}(\overline{x})=x$ and $\overline{x} \in \msc^{-1}(\msc({\mathbf k}^*))$.  The lower commutative triangle is used to show that $\overline{y} \in \msc^{-1}(\msc({\mathbf k}^*))$, while the upper commutative triangle is used to show that $\pi_{p'q'}(\overline{y})=x$.  Hence, $x \in \pi_{p'q'}(\msc^{-1}(\msc({\mathbf k}^*)))$.


To complete the proof, we need only show that the parameter $k_{pq}$ is \uline{not} generically globally identifiable.  Accordingly, let $Z$ denote the hyperplane in $ \mathbb{R}^{|E| + |\Leak|}$ defined by $k_{pq} = k_{p'q'}$ (so, $Z$ is a measure-zero subset of $ \mathbb{R}^{|E| + |\Leak|}$).
Let 
${\mathbf k}^*:= 
(k^*_{i j})_{(j,i) \in \mathcal{P}_{\mathcal{M}}}
\in \mathbb{R}^{|E| + |\Leak|} \smallsetminus Z$.  
Consider 
$\widetilde\sigma({\mathbf k}^*) = 
(k^*_{\sigma(i) \sigma(j)})_{(j,i) \in \mathcal{P}_{\mathcal{M}}}
$, where $\sigma(0):=0$ (to account for leak parameters $k^*_{0j}$).
From the commutative diagram shown earlier in the proof, 
we have 
$\msc({\mathbf k}^*)
= 
\msc(\widetilde \sigma({\mathbf k}^*))$.
Hence, the 
 projected set~\eqref{eq:set-parameters} for $e$, contains both $k^*_{pq}$ and $k^*_{ \sigma(p) \sigma(q)}$ (which are distinct, because ${\mathbf k}^* \notin Z$).
We conclude that, as desired, $k_{qp}$ is \uline{not} generically globally identifiable.
\end{proof}

\section{Results} \label{sec:results}
In this section, we investigate the identifiability of all parameters in mammillary models with one input, one output, and no leaks. 
For ease of notation, we let
$\mathcal{M}_n(i,j)$ denote the $n$-compartment mammillary model $(G, \In, \Out, \Leak)$ with $\In=\{i\}$, $\Out=\{j\}$, and $\Leak=\varnothing$.  For instance, $\mathcal{M}_4(1,2)$ denotes the model shown earlier in Figure~\ref{fig:mammillary}.

Up to symmetry, the mammillary models $\mathcal{M}_n(i,j)$ fall into five families.  These families are defined by the location of the input and the output, so 
we consider the following representatives, one for each of the five families:
\begin{align} \label{eq:5-families}
\mathcal{M}_n(1,1),~ \mathcal{M}_n(1,2), ~\mathcal{M}_n(2,1),~ \mathcal{M}_n(2,2),~ {\rm and}~ \mathcal{M}_n(2,3)~.
\end{align}

\begin{remark} \label{rem:1-family-uniden}
Proposition~\ref{prop:iden-mammillary} implies that, among the models~\eqref{eq:5-families}, 
only the model $\mathcal{M}_n(2,3)$ is unidentifiable or, equivalently, by Remark~\ref{rem:gli-param}, contains unidentifiable parameters. Hence, every parameter in the remaining four models in~\eqref{eq:5-families} is generically locally identifiable (and thus is SLING or generically globally identifiable).
    
\end{remark}

Our main result, which was summarized earlier in Figure \ref{fig:summary}, is as follows.

\begin{theorem}[Main result]
\label{thm:summary} For mammillary models with one input, one output, and no leaks, the parameters have the following identifiability properties:
\begin{itemize}
    \item For all $n \geq 3$, every parameter of $\mathcal{M}_n(1,1)$ is SLING. 
    \item 
    For all $n \geq 4$, every parameter of 
    $\mathcal{M}_n(1,2)$ is SLING, except the parameter $k_{21}$ labeling the edge from the input to output, which is 
    globally identifiable.  
    \item For all $n \geq 4$, every parameter of 
    $\mathcal{M}_n(2,1)$ is SLING, except the parameters $k_{12}$ and $k_{21}$ labeling the edges between the input and output, which are, respectively, globally identifiable and generically globally identifiable.
    \item For all $n \geq 4$, every parameter of 
    $\mathcal{M}_n(2,2)$ is SLING, except the parameters $k_{12}$ and $k_{21}$ labeling the edges between the input and output, which are generically globally identifiable.
    \item For all $n \geq 5$, 
    the parameters 
    of $\mathcal{M}_n(2,3)$
    that label the edges from non-input, non-output peripheral compartments to the central compartment (namely, $k_{14}, k_{15}, \dots, k_{1n}$)  are SLING.
\end{itemize}
\end{theorem}

The remainder of this section is dedicated to proving Theorem~\ref{thm:summary} (see Propositions~\ref{prop:1-1}, 
\ref{prop:1-2},
\ref{prop:2-1},
\ref{prop:2-2a},
and
\ref{prop:2-3} below).  
We also state a conjecture concerning model $\mathcal{M}_n(2,3)$
that, if proven, would elevate Theorem~\ref{thm:summary} to match what is depicted in Figure~\ref{fig:summary} (see Conjecture~\ref{conj:2-3}).  This conjecture is supported by examples we computed using software (see Remark~\ref{rem:sian}).

\begin{remark} \label{rem:in-to-out-edge}
    In Theorem~\ref{thm:summary}, 
    we see each of the models
    $\mathcal{M}_n(1,2)$ and $\mathcal{M}_n(2,1)$ has one parameter that is globally identifiable (which is stronger than being \uline{generically} globally identifiable).  This global identifiability comes from the fact that such a parameter corresponds to the edge from input to output (recall Lemma~\ref{lem:edge-in-to-out-global}).
\end{remark}

\subsection{Preliminaries} \label{sec:prelim}
All five model families~\eqref{eq:5-families} share
the same compartmental matrix $A$ (for a given number of compartments $n$).
Thus, the set of coefficients on the left-hand side of the input-output equation~\eqref{eq:in-out} also is the same.  Accordingly, we give a formula for these coefficients in this subsection (Proposition~\ref{prop:LHS-coefficients-updated} below).
Appearing in this formula (and other formulas for coefficients we see later in this section) are elementary symmetric polynomials.

\begin{definition} \label{def:elem-sym-poly}
Let
$k$ and $n$ be positive integers, with $1 \leq k \leq n$.
Let $X_n$ denote the set of variables $x_1, x_2,\hdots, x_n$. 
The $k$-th \textbf{elementary symmetric polynomial} on $X_n$, which we denote by $e_k(X_n)$, is given by
\begin{align*}
e_k(X_n) ~:=~ \sum_{\substack{I \subseteq [n] \\ |I| = k}} 
    \left( \prod_{i\in I} x_i\right) ~.
\end{align*}
We additionally adopt the following standard convention: 
$e_0(X_n):=1$.
\end{definition}

\begin{example}[Example~\ref{ex:mammillary-solve}, continued] 
    \label{ex:elem-sym-poly}
    We revisit the mammillary model 
    $\mathcal{M}_4(1,2)$ in Figure~\ref{fig:mammillary}. 
    The left-hand side coefficients of the input-output equation (see~\eqref{eq:coeff-running-example}) can be rewritten in terms of elementary symmetric polynomials, as follows:
    \begin{align}
        \label{eq:LHS-coeff-elementary-example}
        \notag
        c_3 ~&=~ \msc_1 ~=~
            e_1(\mathcal{I})
            ~+~
            k_{21} 
            +
            k_{31} 
            +
            k_{41} ~,        
        \\
        c_2 ~&=~ \msc_2 ~=~
            e_2(\mathcal{I})
            ~+~
            k_{21} e_1(\mathcal{I} \smallsetminus \{k_{12}\} )
            +
            k_{31} e_1(\mathcal{I} \smallsetminus \{k_{13}\})
            +
            k_{41} e_1(\mathcal{I} \smallsetminus \{k_{14}\})        \\
        \notag
        c_1 ~&=~ \msc_3 ~=~ 
            e_3(\mathcal{I})
            ~+~
            k_{21} e_2(\mathcal{I} \smallsetminus \{k_{12}\} )
            +
            k_{31} e_2(\mathcal{I} \smallsetminus \{k_{13}\})
            +
            k_{41} e_2(\mathcal{I} \smallsetminus \{k_{14}\})~,
            \end{align}
    where 
    $\mathcal{I} := \{k_{12},k_{13} , k_{14}  \} $ is the set of all ``incoming'' edge labels (that is, labels of edges of the graph $G$ in Figure~\ref{fig:original-graph-G} that are directed toward the central compartment).  
\end{example}

Example~\ref{ex:elem-sym-poly} displayed formulas~\eqref{eq:LHS-coeff-elementary-example} for the left-hand side coefficients for $\mathcal{M}_4(1,2)$.  We give a new formula for these coefficients, using subgraphs of the graph in Figure~\ref{fig:graph-G-tilde}, as follows. 

\begin{figure}
\begin{center}
	\begin{tikzpicture}[scale=1.7]
\draw (3,0) circle (0.2);
 	\draw (3.7,-1.2) circle (0.2);
 	\draw (4.4,-0.3) circle (0.2);
 	\draw (4 ,.8) circle (0.2);
    	\node[] at (3, 0) {1};
    	\node[] at (3.7, -1.2) {2};
     \node[] at (4.2, 0.4) {$\vdots$};
    	\node[] at (4.4, -0.3) {3};
    	\node[] at (4, .8) {$n$};
	 \draw[->] (3.25, -0.05) -- (4.1, -0.3);
	 \draw[<-] (3.25, 0.05) -- (4.1, -0.2);
	 \draw[<-] (3.1, 0.2) -- (3.7, .65);
	 \draw[->] (3.2, .15) -- (3.8, .6);
   	 \node[] at (3.9, .05) {$k_{13}$};
   	 \node[] at (3.25, 0.7) {$k_{1n}$};
   	 \node[] at (3.8, -0.35) {$k_{31}$};
   	 \node[] at (3.9, 0.35) {$k_{n1}$};
	\end{tikzpicture}
\end{center}
\caption{Graph obtained from the star graph in Figure~\ref{fig:original-graph-G} by removing the pair of edges between compartment-$1$ and compartment-$2$.}
    \label{fig:graph-G-tilde}
\end{figure}
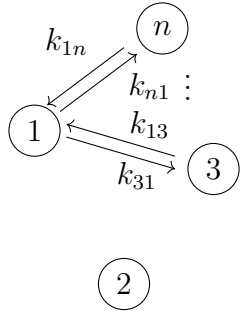

\begin{example}[Example~\ref{ex:elem-sym-poly}, continued] 
    \label{ex:lhs-written-again}
    For the model
    $\mathcal{M}_4(1,2)$, the left-hand side coefficients of the input-output equation (shown in~\eqref{eq:coeff-running-example} and~\eqref{eq:LHS-coeff-elementary-example}) can be rewritten by grouping terms based on whether they contain $k_{21}$, contain $k_{12}$, or contain neither $k_{21}$ nor $k_{12}$:
    \begin{align}
        \label{eq:LHS-coeff-preview}
        \notag
        c_3 ~&=~ k_{21} + k_{12} + g_1
        \\
        c_{2} ~&=~ 
        k_{21}~e_{1}(\Sigma) + k_{12} g_1 + g_2
        \\    
        c_{1} ~&=~ 
            k_{21}~e_{2}(\Sigma) + k_{12} g_2~,
        \notag
    \end{align}
where 
\begin{align*}
\Sigma~:=~\{k_{13},k_{14}\}~,
\quad 
g_1~:=~k_{13} + k_{14} + k_{31} + k_{41}~,
\quad 
{\rm and}
\quad 
g_2~:=~ k_{13}k_{14} + k_{13}k_{41} + k_{14}k_{31}~.
\end{align*}
Notice that $g_1$ is the sum over all edges in the $n=4$ version of the graph in Figure~\ref{fig:graph-G-tilde}, and $g_2$ is the sum over pairs of edges in Figure~\ref{fig:graph-G-tilde} that do not form a cycle.
Also, it is straightforward to deduce the following equation from the equations~\eqref{eq:LHS-coeff-preview}:
    \begin{align}
        \label{eq:n=4-big-sum}
        k_{12}^2 c_3
        -
        k_{12} c_2
        +
        c_1
        ~=~
        (k_{12}^2 - k_{12} e_1(\Sigma) + e_2(\Sigma) ) k_{21}  + k_{12}^3~.
    \end{align}
Next, alternate formulas for $g_1$ and $g_2$ are given by:
\begin{equation} \label{eq:formula-for-g-example}
\begin{pNiceArray}{c}
g_1 \\
g_2 
\end{pNiceArray}
    ~=~ 
\begin{pNiceArray}{c}
e_1(\Sigma) \\
e_2(\Sigma) 
\end{pNiceArray}
+
\begin{pNiceArray}{cc}
1 & 1\\
k_{14} & k_{13}
\end{pNiceArray}
\begin{pNiceArray}{c}
k_{31} \\
k_{41} 
\end{pNiceArray}~.
\end{equation}
The formulas~\eqref{eq:LHS-coeff-preview} and equations~\eqref{eq:n=4-big-sum}--\eqref{eq:formula-for-g-example} generalize 
to what is shown in~\eqref{eq:lhs-updated},~\eqref{eq:big-sum}, and~\eqref{eq:formula-for-g}, respectively, in the next result.
\end{example}

\begin{proposition}[Left-hand side coefficients]
    \label{prop:LHS-coefficients-updated}
    Let  $n \geq 3$. 
    Consider an $n$-compartment mammillary model $\mathcal{M} =  (G, \In, \Out, \Leak)$ with $\lvert \In \rvert = \lvert \Out \rvert = 1$ and $\Leak =\varnothing$. 
    Let $c_0,c_1,\dots, c_{n-1}$ denote the coefficients of the left-hand side of the input-output equation, as in~\eqref{eq:eq:i-o-c-and-d}.  Then the following hold:
    \begin{itemize}
        \item 
    $c_0=0$, and the 
    coefficients $c_1,c_2,\dots, c_{n-1}$ are given by the following formulas:
\begin{align} \label{eq:lhs-updated}
    \notag
c_{n-1} ~&=~ 
        k_{21} + k_{12} g_0 + g_1\\
    \notag
    c_{n-2} ~&=~ 
        k_{21}~e_{1}(\Sigma) + k_{12} g_1 + g_2\\    
    c_{n-3} ~&=~ 
            k_{21}~e_{2}(\Sigma) + k_{12} g_2 + g_3\\    
    \notag
    &~ \vdots \\
    \notag
    c_{1} ~&=~ 
            k_{21}~e_{n-2}(\Sigma) + k_{12} g_{n-2} + g_{n-1} ~,
\end{align}
where $\Sigma:=\{k_{13}, k_{14}, \dots, k_{1n} \}$ and, for $i=0,1,2,\dots, n-1$, we denote:
    \begin{align} \label{eq:g_i-equation}
    g_i ~:=~
    \sum_{F \mathrm{~is~an~}
        i\mathrm{-edge~spanning,~incoming~forest~of~} \widetilde{G}
        } \pi_F~,
    \end{align}
    where $\widetilde{G}$ denotes the graph in Figure~\ref{fig:graph-G-tilde}. 
    \item The following equation holds:
    \begin{align} \label{eq:big-sum}
        \notag
        & k_{12}^{n-2} c_{n-1}
        -
        k_{12}^{n-3} c_{n-2}
        +
        k_{12}^{n-4} c_{n-3}
        -
        \dots
        \pm
         c_{1} 
        \\
        & \quad\quad  ~=~
        \left(
        k_{12}^{n-2}
        -
        k_{12}^{n-3}~ e_1(\Sigma)
        +
        k_{12}^{n-4}~ e_2(\Sigma)
        -
        \dots
        \pm
         e_{n-2}(\Sigma)
        \right) k_{21}
        ~+~
        k_{12}^{n-1}~.
    \end{align}
    \item The terms $g_i$ are given by $g_0=1$, $g_{n-1}=0$, and the following formula: 
\begin{equation} \label{eq:formula-for-g}
\begin{pNiceArray}{c}
g_1 \\
g_2 \\
\vdots \\
g_{n-2}
\end{pNiceArray}
    ~=~ 
\begin{pNiceArray}{c}
e_1(\Sigma) \\
e_2(\Sigma) \\
\vdots \\
e_{n-2}(\Sigma)
\end{pNiceArray}
+
M
\begin{pNiceArray}{c}
k_{31} \\
k_{41} \\
\vdots \\
k_{n1}
\end{pNiceArray}~,
\end{equation}
%
where $\Sigma:=\{k_{13},k_{14},\dots,k_{1n}\}$ is the set of all ``incoming'' edge labels of $\widetilde{G}$ (the graph in Figure~\ref{fig:graph-G-tilde}), and where $M$ denotes the following $(n-2) \times (n-2)$ matrix:
\begin{equation} \label{eq:matrix-for-G-tilde}
    M~=~
    \begin{pNiceArray}{cccc}
    1 & 1 & \dots & 1
    \\
e_{1}(\Sigma \smallsetminus \{{k}_{13}\}) & 
    e_{1}(\Sigma \smallsetminus \{{k}_{14}\}) & 
    \dots
    &
    e_{1}(\Sigma \smallsetminus \{{k}_{1n}\}) \\
e_{2}(\Sigma \smallsetminus \{{k}_{13}\}) & 
    e_{2}(\Sigma \smallsetminus \{{k}_{14}\}) & 
    \dots
    &
    e_{2}(\Sigma \smallsetminus \{{k}_{1n}\}) \\
\vdots & \vdots & \ddots & \vdots \\
e_{n-3}(\Sigma \smallsetminus \{{k}_{13} \}) & 
    e_{n-3}(\Sigma \smallsetminus \{{k}_{14}\}) & 
    \dots
    &
    e_{n-3}(\Sigma \smallsetminus \{{k}_{1n}\}) 
\end{pNiceArray}~.
\end{equation}
    \item The determinant of this matrix $M$ is, up to sign, the {\em Vandermonde polynomial} $\prod_{3 \leq j < \ell \leq n} (k_{1j}- k_{1 \ell})$.
    \end{itemize}
\end{proposition}

\begin{proof}
    The equality $c_0=0$ was given earlier (Remark~\ref{rem:c0=0}).    Next, for $i=1,2,\dots, n-1$, recall from 
Proposition~\ref{prop:coeff-formula} that 
the coefficient $c_i$ is the sum of all terms $\pi_F$, where $F$ is 
    a spanning incoming forest that is an $(n-i)$-edge subgraph of the star graph $G$ (in Figure~\ref{fig:original-graph-G}).  Such a forest can not contain both of the edges $k_{21}$ and $k_{12}$ (to avoid a cycle), so each $c_i$ is a sum over three types of forests:
    \begin{itemize}
        \item \uline{forests involving $k_{21}$ (but not $k_{12}$)} -- in this case, as $k_{21}$ is outgoing from compartment-$1$, the remaining edges must be chosen from the ``incoming'' edges (those directed toward compartment-$1$) from $3, 4, \dots, n$ (and can be arbitrarily chosen); these edges are indexed by the set $\Sigma$;
        \item \uline{forests involving $k_{12}$ (but not $k_{21}$)} -- in this case, the edge $k_{12}$ can be combined with any set of edges of $\widetilde G$ that do not contain a cycle;
        \item \uline{forests involving neither $k_{12}$ nor $k_{21}$} -- these are spanning, incoming forests of $\widetilde{G}$. 
    \end{itemize}
Viewing each $c_i$ as such a sum exactly yields the formulas~\eqref{eq:lhs-updated}, as claimed.

Next, we note that
\begin{align} \label{eq:g-boundary-conditions}
g_0~=~1 \quad \quad  {\rm and} \quad \quad g_{n-1}=0~,    
\end{align}
because (1) we have $\pi_F=1$ for the subgraph $F$ with no edges, and (2) spanning, incoming forests of $\widetilde{G}$ have at most $n-2$ edges (and hence none have $n-1$ edges).

Next, it is straightforward to use the equations~\eqref{eq:lhs-updated} and~\eqref{eq:g-boundary-conditions} to obtain the equation~\eqref{eq:big-sum}.
To provide some detail, we use~\eqref{eq:lhs-updated} to obtain expressions for $k_{12}^{n-2} c_{n-1}$, $
        k_{12}^{n-3} c_{n-2}$, 
        \dots ,
        $
         c_{1} $, and then we observe that 
the terms involving $g_i$'s ``telescope'' as they are added together, so that subsequently what remains of these terms is $k_{12}^{n-1}g_0 + g_{n-1} $, which by~\eqref{eq:g-boundary-conditions}, equals $k_{12}^{n-1}$.

We next prove the formula for the $g_i$'s given in~\eqref{eq:formula-for-g}.  Recall that $g_i$ is the sum of all terms $\pi_F$, where $F$ is 
    a spanning incoming forest that is an $i$-edge subgraph of the star graph $\widetilde G$ in Figure~\ref{fig:graph-G-tilde}.  By definition, such forests contain at most $1$ ``outgoing'' edge (i.e., an edge of the form $1 \to j$).  Now it is straightforward to check that  $e_{i}(\Sigma)$ is the subsum of $g_i$ corresponding to forests with no ``outgoing'' edges, and the remaining terms on the right-hand side of~\eqref{eq:formula-for-g} form the subsum corresponding to forests with exactly $1$ ``outgoing'' edge (a similar idea underlies the enumeration of spanning incoming forests shown in Figure~\ref{fig:spanning-inc-forest-G}).  

Finally, the determinant of the matrix $M$ is well known (see equation ~(6) in the proof of~\cite[Theorem~6.1]{singularlocus}).
\end{proof}

We end this subsection by recalling a useful fact about elementary symmetric polynomials; this result is well known (cf.~\cite[Proof of Theorem 6.3]{singularlocus}).
 
\begin{lemma} \label{lem:deg-elem-sym}
   Consider the following map, which evaluates a vector $(x_1,x_2,\dots,x_n)$ at all elementary symmetric polynomials on the set $\{x_1,x_2,\dots,x_n\}$:
    \begin{align*}
   \mathbb{R}^n ~&\to~ \mathbb{R}^n 
\\
(x_1,x_2,\dots,x_n) ~&\mapsto~
\left( e_1(\{x_1,x_2,\dots,x_n\}),~ e_2(\{x_1,x_2,\dots,x_n\}),~\dots, e_n(\{x_1,x_2,\dots,x_n\})\right) ~.
    \end{align*}   
Then this map is generically $n!$-to-$1$.
\end{lemma}

\subsection{Identifiability of \texorpdfstring{$\mathcal{M}_n(1,1)$}{Mn(1,1)}}

In this subsection, we show that all parameters are SLING in the mammillary model with input and output in the central compartment (see Figure~\ref{fig:1-1}).

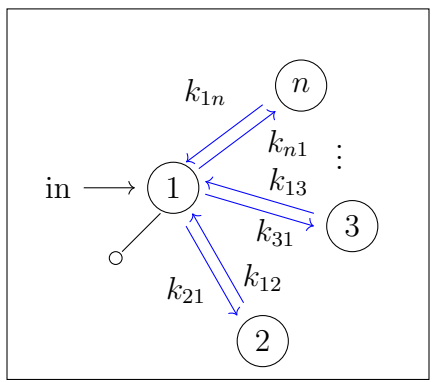
\begin{figure}[ht]
\begin{center}
	\begin{tikzpicture}[scale=1.7]
\draw (3,0) circle (0.2);
 	\draw (3.7,-1.2) circle (0.2);
 	\draw (4.4,-0.3) circle (0.2);
 	\draw (4 ,.8) circle (0.2);
    	\node[] at (3, 0) {1};
    	\node[] at (3.7, -1.2) {2};
     \node[] at (4.3, 0.3) {$\vdots$};
    	\node[] at (4.4, -0.3) {3};
    	\node[] at (4, .8) {$n$};
     \draw[->][blue] (3.1, -.3) -- (3.5, -1);
	 \draw[<-][blue](3.15, -.2) -- (3.55, -0.9);
	 \draw[->][blue] (3.25, -0.05) -- (4.1, -0.3);
	 \draw[<-][blue] (3.25, 0.05) -- (4.1, -0.2);
	 \draw[->][blue] (3.2, .15) -- (3.8, .6);
	 \draw[<-][blue] (3.1, 0.2) -- (3.7, .65);
   	 \node[] at (3.1, -.8) {$k_{21}$};
   	 \node[] at (3.7, -0.7) {$k_{12}$};
   	 \node[] at (3.9, .05) {$k_{13}$};
   	 \node[] at (3.8, -0.35) {$k_{31}$};
   	 \node[] at (3.25, 0.75) {$k_{1n}$};
   	 \node[] at (3.9, 0.35) {$k_{n1}$};
 	\draw (2.55,-.55) circle (0.05);	
	 \draw[-] (2.9, -.2 ) -- (2.6, -.5);	
  
	 \draw[->] (2.3, 0) -- (2.7, 0);	
   	 \node[] at (2.1, 0) {in};

\draw (1.7,-1.5) rectangle (5., 1.4);
	\end{tikzpicture}
\end{center}
    \caption{Parameter identifiability for  $\mathcal{M}_n(1,1)$: all parameters are SLING.}
    \label{fig:1-1}
\end{figure}


\begin{proposition}[$\mathcal{M}_n(1,1)$] \label{prop:1-1}
Let $n \geq 3$.
If $ \mathcal M =(G, \In, \Out, \Leak)$
is an $n$-compartment mammillary model 
with no leaks ($\Leak = \varnothing$) and input and output in the central compartment ($\In=\Out=\{1\}$), then \uline{every} parameter of $\mathcal{M}$ is SLING.
\end{proposition}

\begin{proof}
Every permutation of $\{2,3,\dots, n\}$ is a model automorphism of $\mathcal{M}$.  Thus, the ``incoming'' edges $k_{12}, k_{13}, \dots, k_{1n}$ are symmetrically identical, as are the ``outgoing'' edges $k_{21}, k_{31}, \dots, k_{n1}$.  Hence (and here $n \geq 3$ is used),  Lemma~\ref{lem:symmetry} implies that every parameter is SLING or unidentifiable.  
However, Remark~\ref{rem:1-family-uniden} implies that all parameters are at least locally identifiable.  We conclude that all parameters are SLING.
\end{proof}

\begin{remark}
    Proposition~\ref{prop:1-1} was proven previously~\cite{cobelli-lepschy-romaninjacur} (using other methods), but we provide a proof here for completeness.
\end{remark}


\subsection{Identifiability of \texorpdfstring{$\mathcal{M}_n(1,2)$}{Mn(1,2)}}

For the $4$-compartment model $\mathcal{M}_4(1,2)$, from Figure~\ref{fig:mammillary}, we saw 
in Example~\ref{ex:mammillary-solve}
that the parameter $k_{21}$ is globally identifiable and all other parameters are SLING.  This result generalizes to allow for more compartments, as follows (see Figure~\ref{fig:1-2}).

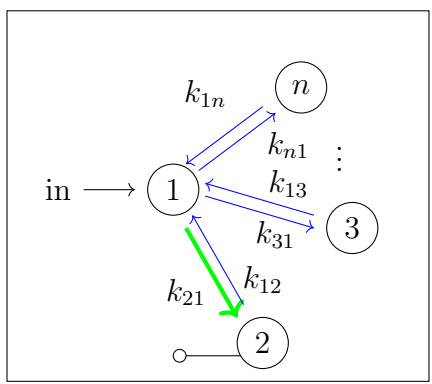
\begin{figure}[ht]
\begin{center}
	\begin{tikzpicture}[scale=1.7]
\draw (3,0) circle (0.2);
 	\draw (3.7,-1.2) circle (0.2);
 	\draw (4.4,-0.3) circle (0.2);
 	\draw (4 ,.8) circle (0.2);
    	\node[] at (3, 0) {1};
    	\node[] at (3.7, -1.2) {2};
     \node[] at (4.3, 0.3) {$\vdots$};
    	\node[] at (4.4, -0.3) {3};
    	\node[] at (4, .8) {$n$};
     \draw[->][ultra thick, green](3.1, -.3) -- (3.5, -1);
	 \draw[<-][blue](3.15, -.2) -- (3.55, -0.9);
	 \draw[->] [blue](3.25, -0.05) -- (4.1, -0.3);
	 \draw[<-] [blue](3.25, 0.05) -- (4.1, -0.2);
	 \draw[->][blue] (3.2, .15) -- (3.8, .6);
	 \draw[<-][blue] (3.1, 0.2) -- (3.7, .65);
   	 \node[] at (3.1, -.8) {$k_{21}$};
   	 \node[] at (3.7, -0.7) {$k_{12}$};
   	 \node[] at (3.9, .05) {$k_{13}$};
   	 \node[] at (3.8, -0.35) {$k_{31}$};
   	 \node[] at (3.25, 0.75) {$k_{1n}$};
   	 \node[] at (3.9, 0.35) {$k_{n1}$};
 	\draw (3.05,-1.3) circle (0.05);	
	 \draw[-] (3.53, -1.3 ) -- (3.1, -1.3);	
  
	 \draw[->] (2.3, 0) -- (2.7, 0);	
   	 \node[] at (2.1, 0) {in};

\draw (1.7,-1.5) rectangle (5., 1.4);
	\end{tikzpicture}
\end{center}
    \caption{Parameter identifiability for $\mathcal{M}_n(1,2)$: the (green) boldface arrow indicates a globally identifiable parameter, and all others are SLING.
    }
    \label{fig:1-2}
\end{figure}

\begin{proposition}[$\mathcal{M}_n(1,2)$] \label{prop:1-2}
Let $n \geq 4$.
If $ \mathcal M =(G, \In, \Out, \Leak)$
is an $n$-compartment mammillary model 
with 
$\In=\{1\}$, 
$\Out=\{2\}$, and 
$\Leak = \varnothing$, then 
    \begin{enumerate}
        \item the parameter $k_{21}$ is globally identifiable, 
        \item the parameter $k_{12}$ is SLING, and
        \item the parameters $k_{13}, k_{14}, \dots, k_{1n}, ~ k_{31}, k_{41}, \dots, k_{n1}$ are SLING. 
    \end{enumerate}
\end{proposition}

\begin{proof} 
Part~(1) follows directly from Lemma~\ref{lem:edge-in-to-out-global}.  Also, part~(3) follows from 
Remark~\ref{rem:1-family-uniden}, 
Lemma~\ref{lem:symmetry}, 
and the fact that every permutation of $\{3,4,\dots,n\}$ is a model automorphism of $\mathcal{M}$ (here, the assumption $n \geq 4$ is used).

For part~(2), 
we begin by analyzing the coefficients on the right-hand side of the input-output equation.  By Proposition~\ref{prop:coeff-formula}, these coefficients (denoted by $d_i$) are sums over certain subgraphs of the following graph $G_2^*$, which we obtain from the star graph in Figure~\ref{fig:original-graph-G} by removing the (unique) edge outgoing from  compartment-$2$ (the output compartment):

\begin{center}
	\begin{tikzpicture}[scale=1.6]
\draw (3,0) circle (0.2);
 	\draw (3.7,-1.2) circle (0.2);
 	\draw (4.4,-0.3) circle (0.2);
 	\draw (4 ,.8) circle (0.2);
    	\node[] at (3, 0) {1};
    	\node[] at (3.7, -1.2) {2};
        \node[] at (4.2, 0.4) {$\vdots$};
    	\node[] at (4.4, -0.3) {3};
    	\node[] at (4, .8) {$n$};
     \draw[->] (3.1, -.3) -- (3.5, -1);
	 \draw[->] (3.25, -0.05) -- (4.1, -0.3);
	 \draw[<-] (3.25, 0.05) -- (4.1, -0.2);
	 \draw[->] (3.2, .15) -- (3.8, .6);
	 \draw[<-] (3.1, 0.2) -- (3.7, .7);
   	 \node[] at (3.1, -.8) {$k_{21}$};
   	 \node[] at (3.9, .05) {$k_{13}$};
   	 \node[] at (3.8, -0.35) {$k_{31}$};
   	 \node[] at (3.25, 0.75) {$k_{1n}$};
   	 \node[] at (3.9, 0.35) {$k_{n1}$};
	\end{tikzpicture}
\end{center}
Such subgraphs must connect the input and output compartments (in compartment-$1$ and compartment-$2$, respectively) and so must contain the edge $k_{21}$.  Next, the condition of being an incoming forest precludes the remaining outgoing edges from compartment-$1$ (namely, $k_{31},k_{41},\dots, k_{n1}$).  Thus, the coefficients $d_i$ are given by the following formulas involving elementary symmetric polynomials on the set $\Sigma:=\{k_{13}, k_{14}, \dots, k_{1n} \}$:
\begin{align} \label{eq:rhs-coeff-proof-12}
    \notag
    d_{n-2} ~&=~ k_{21}\\
    \notag
    d_{n-3} ~&=~ k_{21}~ e_1(\Sigma) \\
    d_{n-4} ~&=~ k_{21}~ e_2(\Sigma) \\
    \notag
     & ~ \vdots \\
     \notag
    d_0 ~&=~ k_{21}~ e_{n-2}(\Sigma) 
    ~. 
\end{align}

Next, Proposition~\ref{prop:LHS-coefficients-updated} gives the following equation, from~\eqref{eq:big-sum}, involving left-hand side coefficients $c_i$, the parameters $k_{12}$ and $k_{21}$, and elementary symmetric polynomials on $\Sigma$:
\begin{align} \label{eq:big-sum-in-proof-1-2}
        \notag
        & k_{12}^{n-2} c_{n-1}
        -
        k_{12}^{n-3} c_{n-2}
        +
        k_{12}^{n-4} c_{n-3}
        -
        \dots
        \pm
         c_{1} 
        \\
        & \quad\quad  ~=~
        \left(
        k_{12}^{n-2}
        -
        k_{12}^{n-3}~ e_1(\Sigma)
        +
        k_{12}^{n-4}~ e_2(\Sigma)
        -
        \dots
        \pm
         e_{n-2}(\Sigma)
        \right) k_{21}
        ~+~
        k_{12}^{n-1}~.
    \end{align}
By moving all terms of equation~\eqref{eq:big-sum-in-proof-1-2} to the right-hand side, we obtain the following equation:
\begin{align} \label{eq:big-sum-in-proof-1-2-moved-to-rhs}
        \notag
        0 ~&=~
        k_{12}^{n-1}
        +
        ( k_{21} -  c_{n-1}) 
	k_{12}^{n-2}
        -
        ( k_{21} e_1(\Sigma) -c_{n-2})
        k_{12}^{n-3} 
        +
        \dots
        \pm
         ( k_{21} e_{n-2}(\Sigma) - c_{1})
         \\
         &=~ 
         k_{12}^{n-1}
        +
         \sum_{i=2}^{n} (-1)^i ( k_{21} e_{i-2}(\Sigma) -c_{n-i+1}) k_{12}^{n-i}
         ~,
    \end{align}
where we use the fact that $e_0(\Sigma):=1$. 
Our aim is to apply Proposition~\ref{prop:polynom-to-iden} using the following function which is closely related to the right-hand side of~\eqref{eq:big-sum-in-proof-1-2-moved-to-rhs}: \begin{align} \label{eq:h-in-proof-1-2}
	h
         &~:=~ 
         z^{n-1}
        +
         \sum_{i=2}^{n} (-1)^i ( k_{21} e_{i-2}(\Sigma) -c_{n-i+1}) z^{n-i}
         ~.
    \end{align}
By construction (see equations~\eqref{eq:big-sum-in-proof-1-2-moved-to-rhs}--\eqref{eq:h-in-proof-1-2}), $h|_{z=k_{12}} = 0$.  
Also, the leading coefficient (with respect to $z$) is~$1$, which is a nonzero function that is identifiable from the coefficient map (as in Definition~\ref{def:identifiable-function}). 
Therefore, to apply Proposition~\ref{prop:polynom-to-iden}, we need only show that for all $i=2,3,\dots, n$, the polynomial function $(-1)^i ( k_{21} e_{i-2}(\Sigma) - c_{n-i+1})$ is identifiable from the coefficient map.  Indeed, the coefficients $c_i$ are (by definition) identifiable from the coefficient map, and we see from~\eqref{eq:rhs-coeff-proof-12} that $k_{21}$ and the elementary symmetric polynomials $e_1(\Sigma),e_2(\Sigma),\dots, e_{n-2}(\Sigma)$ are identifiable from the coefficient map.
Thus, we indeed can apply Proposition~\ref{prop:polynom-to-iden} and thereby conclude that $k_{12}$ is generically locally identifiable.

It remains only to show that $k_{12}$ is \emph{not} generically \uline{globally} identifiable.  By Definition~\ref{def:iden-param}, it suffices to find a positive-measure subset $\Omega$ of the parameter space $\mathbb{R}^{\lvert E\rvert + \lvert \Leak\rvert}
=
\mathbb{R}^{2n}$, such that, on $\Omega$, the relevant 
set arising from the coefficient map (namely,~\eqref{eq:set-parameters}) has size at least two.  

As a step toward defining such a set $\Omega$, we revisit the function $h$, from~\eqref{eq:h-in-proof-1-2}.  We rewrite $h$ by using the equations~\eqref{eq:lhs-updated} in straightforward manner to obtain the first equality here:
\begin{align} \label{eq:h-in-proof-1-2-rewritten}
    \notag
	h
         &~=~ 
         z^{n-1}
        +
         \sum_{i=2}^{n} (-1)^{i+1} ( k_{12} g_{i-2} + g_{i-1} ) z^{n-i}
         ~\\
         &~=~ 
        (z - k_{12}) ( z^{n-2} - g_1 z^{n-3} + g_2 z^{n-4}- \dots \pm g_{n-2}  )
        ~,
    \end{align}
and the second equality above uses the fact that $g_0=g_{n-1}=0$ (from Proposition~\ref{prop:LHS-coefficients-updated}).

Our next aim is to show that there is an open subset of parameter space $\mathbb{R}^{2n}$ on which, informally speaking, $h$ always specializes to a (univariate) polynomial with $n-1$ distinct real roots.  More precisely, we will show the existence of an open set $\Theta \subseteq \mathbb{R}^{2n}$ such that, for all $\mathbf{k}^* \in \Theta$, the polynomial $h|_{\mathbf{k}=\mathbf{k}^*}$ (which is univariate in $z$) has $n-1$ distinct real roots.

We will construct such a set $\Theta$ as an open neighborhood of some vector $\mathbf{k^{\bullet}}$.  To this end, we choose a vector of parameters 
$\mathbf{k^{\bullet}} = 
(
k^{\bullet}_{12}, 
k^{\bullet}_{21},
k^{\bullet}_{13}, 
k^{\bullet}_{31},
\dots,
k^{\bullet}_{1n}, 
k^{\bullet}_{n1}
) \in \mathbb{R}^{2n}
$ with two properties:
    \begin{enumerate}[label=(\roman*)]
        \item 
    $k^{\bullet}_{31}=k^{\bullet}_{41}=\dots=k^{\bullet}_{n1}=0$ 
(that is, all ``outgoing'' parameters, except $k_{21}$, are $0$), and
        \item the ``incoming'' parameters  
$k^{\bullet}_{12}, k^{\bullet}_{13}, 
k^{\bullet}_{14}, 
\dots,
k^{\bullet}_{1n}$ are distinct.
    \end{enumerate}

We claim that $g_i|_{\mathbf{k}=\mathbf{k^{\bullet}}} = e_i(\Sigma^{\bullet})$ for all $i=1,2,\dots, n-2$, where 
$\Sigma^{\bullet}:=\{k^{\bullet}_{13}, k^{\bullet}_{14}, \dots, k^{\bullet}_{1n} \}$.  
We verify this claim, as follows.
First, recall from~\eqref{eq:g_i-equation} that $g_i$ is a sum over incoming forests of the graph $\widetilde{G}$ in Figure~\ref{fig:graph-G-tilde}. Thus, evaluating at $\mathbf{k}^{\bullet}$ essentially disregards all ``outgoing'' edges in $\widetilde{G}$ (because the corresponding parameters are set to $0$ by property (i)), leaving only ``incoming'' edges.  These ``incoming'' edges correspond to the labels in $\Sigma^{\bullet}$, which can be chosen freely to form incoming forests.  We conclude that $g_i|_{\mathbf{k}=\mathbf{k^{\bullet}}} = e_i(\Sigma^{\bullet})$, as claimed.

Using~\eqref{eq:h-in-proof-1-2-rewritten}, it now follows that:
    \begin{align} \label{eq:see-roots}
        h|_{\mathbf{k}=\mathbf{k^{\bullet}}}
        ~=~
           (z - k^{\bullet}_{12}) ( z^{n-2} - e_1(\Sigma^{\bullet}) z^{n-3} + e_2(\Sigma^{\bullet}) z^{n-4}- \dots \pm e_{n-2}(\Sigma^{\bullet})  )~.
    \end{align}
 It is easily seen from~\eqref{eq:see-roots} that the set of roots of 
 the polynomial 
 $h|_{\mathbf{k}=\mathbf{k^{\bullet}}}$  is $\{k^{\bullet}_{12}\} \cup \Sigma^{\bullet}$ and therefore $h|_{\mathbf{k}=\mathbf{k^{\bullet}}}$ has $n-1$ distinct real roots (property (ii) of $k^{\bullet}$ is used here).  
 We claim that there is an open neighborhood $\Theta$ of $\mathbf{k^{\bullet}}$ such that for every $\mathbf{k^{*}} \in \Theta$, the polynomial $h|_{\mathbf{k}={\mathbf{k^{*}}}}$ has $n-1$ distinct real roots.  Indeed, this follows readily from two well-known facts: (1)~the roots of a polynomial vary continuously in the coefficients of the polynomial (see, e.g.,~\cite[Theorem~1.3.1, page 10]{Rahman:Schmeisser:AnalyticTheoryOfPolynomials}), and (2)~non-real roots of real-coefficient univariate polynomials come in complex-conjugate pairs.


%

Next, if needed, we shrink $\Theta$ to a smaller open neighborhood $\Omega$ 
of $\mathbf{k}^{\bullet}$
so that property~(ii) holds in the neighborhood.  More precisely, we ensure that, for all $\mathbf{k^{*}} $ in the neighborhood~$\Omega$, the parameter values $k^{*}_{12}, k^{*}_{13}, 
\dots,
k^{*}_{1n}$ are distinct.  

Now consider an arbitrary parameter vector $\mathbf{k^{*}} $ in 
$\Omega$.  
By Definition~\ref{def:iden-param}, it suffices to show the existence of another parameter vector $\mathbf{k}^{**} \in \mathbb{R}^{2n}
$, with $k^{**}_{12} \neq k^*_{12}$, such that both parameter vectors have the same image under the coefficient map, that is, 
\begin{align} \label{eq:proof-12-SLING}
c_i(\mathbf{k}^{*})~=~c_i(\mathbf{k}^{**}) 
\quad
{\rm and}
\quad 
d_j(\mathbf{k}^{*})~=~d_j(\mathbf{k}^{**})     
\end{align}
for all $1 \leq i \leq n-1$ and $0 \leq j \leq n-2$, 
where the $c_i$'s and $d_j$'s are
as
in~\eqref{eq:lhs-updated} and~\eqref{eq:rhs-coeff-proof-12}, respectively.

To this end, we begin to define such a vector $\mathbf{k}^{**}$ by setting some of its coordinates equal to the corresponding values in $\mathbf{k}^{*}$, 
as follows: 
    \begin{align} \label{eq:coefficients-unchanged}
        k^{**}_{21}~:=~k^*_{21} \quad \textrm{and}
    \quad 
    (k^{**}_{13}, k^{**}_{14}, \dots, k^{**}_{1n})
    ~:=~ ( k^{*}_{13},k^{*}_{14}, \dots, k^{*}_{1n})~.
    \end{align}
The above choice for $k^{**}_{21}$ is mandatory, as $k_{21}$ is globally identifiable.  Also, the choices~\eqref{eq:coefficients-unchanged} guarantee that the parameters appearing in the formulas for $d_0,d_1,\dots,d_{n-2}$, in~\eqref{eq:rhs-coeff-proof-12}, are the same in $\mathbf{k}^{*}$ and $\mathbf{k}^{**}$.  Therefore, 
the desired equalities
$d_j(\mathbf{k}^{*}) = d_j(\mathbf{k}^{**})$, from~\eqref{eq:proof-12-SLING}, hold for all $0 \leq j \leq n-2$.

The rest of this proof consists of defining the remaining coordinates of $\mathbf{k}^{**}$ and then verifying the desired equalities
$c_i(\mathbf{k}^{*})=c_i(\mathbf{k}^{**}) $ from~\eqref{eq:proof-12-SLING}.

To this end, consider the univariate polynomial 
$h|_{\mathbf{k}=\mathbf{k^{*}}}$.  One of its roots, by~\eqref{eq:h-in-proof-1-2-rewritten}, is $k^*_{12}$.  Additionally, there are $n-2$ additional (distinct) real roots (by construction of $\Omega$).  Accordingly, we pick one of these roots and set  
$k^{**}_{12}$ equal to it.  Hence, we have:
\begin{align} \label{eq:choice-of-k12}
h|_{\mathbf{k}=\mathbf{k^{*}}}(k^{**}_{12})=0 \quad \textrm{and}
    \quad 
    k^{**}_{12} \neq k^*_{12}~.
\end{align}

To define the remaining entries of $\mathbf{k}^{**}$, we need to use a matrix that is obtained by specializing and performing row-operations on the matrix $M$ in~\eqref{eq:matrix-for-G-tilde}.  To give details, consider the following $(n-2) \times (n-2)$ matrix:
    \begin{align} \label{eq:M*}
        M^* \quad := \quad M|_{k_{13}=k^*_{13},~k_{14}=k^*_{14},~\dots,~ k_{1n}=k^*_{1n}}~.
    \end{align}
Next, let $\textrm{Shift}(M^*)$ be the  $(n-2) \times (n-2)$ matrix obtained from $M^*$ by shifting the rows down by one; more precisely, the first row of $\textrm{Shift}(M^*)$ consists of $0$'s, the second row is the first row of $M^*$, 
the third row is the second row of $M^*$, 
and so on, ending with 
row-$(n-2)$ being the row-$(n-3)$ of $M^*$.  
Finally, define:
    \begin{align} \label{eq:M-tilde}
        \widetilde{M} \quad := \quad 
            M^* + k^*_{12} \textrm{Shift}(M^*)~.
    \end{align}

We claim that 
$\widetilde{M}$ is invertible.  To see this, recall from Proposition~\ref{prop:LHS-coefficients-updated} that $\det M = \pm \prod_{3 \leq j < \ell \leq n} (k_{1j}- k_{1 \ell})$.  Therefore, by~\eqref{eq:M*},
$\det M^* = \pm \prod_{3 \leq j < \ell \leq n} (k^*_{1j}- k^*_{1 \ell})$, which is nonzero (recall that the fact that $\mathbf{k^{*}}$ is in $ \Omega$ guarantees that $k^{*}_{13}, k^{*}_{14}, 
\dots,
k^{*}_{1n}$ are distinct).  Hence, the matrix  $M^*$ is invertible.  Finally, by construction~\eqref{eq:M-tilde}, $\widetilde{M}$ is obtained from $M^*$ by row operations and so is also invertible.

As $\widetilde{M}$ is invertible, we can now define the remaining entries of $\mathbf{k}^{**}$, as follows, where we use the notation
$\Sigma^*:=\{k^{*}_{13}, k^{*}_{14}, \dots, k^{*}_{1n}\}$:
\begin{equation} \label{eq:how-to-define-remaining-parameters}
    \begin{pNiceArray}{c}
    k^{**}_{31} \\
    k^{**}_{41} \\
    \vdots \\
    k^{**}_{n1}
    \end{pNiceArray}
    \quad := \quad 
    (\widetilde{M})^{-1}
    \begin{pNiceArray}{c}
        c_{n-1}(\mathbf{k}^{*}) - k^*_{21} -k_{12}^{**} - e_1(\Sigma^*) \\
        c_{n-2}(\mathbf{k}^{*}) - k^*_{21} e_1(\Sigma^*)-k_{12}^{**}e_1(\Sigma^*) - e_2(\Sigma^*) \\
        c_{n-3}(\mathbf{k}^{*}) - k^*_{21} e_2(\Sigma^*)-k_{12}^{**}e_2(\Sigma^*) - e_3(\Sigma^*) \\
	\vdots
    \\
        c_{2}(\mathbf{k}^{*}) - k^*_{21} e_{n-3}(\Sigma^*) -k_{12}^{**}e_{n-3}(\Sigma^*) - e_{n-2}(\Sigma^*) \\
    \end{pNiceArray}
    ~.
\end{equation}
 
Next, we verify the equalities $c_i(\mathbf{k}^{*})=c_i(\mathbf{k}^{**}) $, in~\eqref{eq:proof-12-SLING}, for $2 \leq i \leq n-1$.  Indeed, it is straightforward (although tedious) to check that these equalities  
follow readily from the formulas for $c_2,c_3,\dots, c_{n-1}$ in~\eqref{eq:lhs-updated}, the formulas for the $g_i$'s in terms of the matrix $M$ in~\eqref{eq:formula-for-g}, the construction of the matrix $\widetilde{M}$ in~\eqref{eq:M-tilde}, and our choice of parameters $k_{ij}^{**}$, including the formula~\eqref{eq:how-to-define-remaining-parameters}.

Finally, it remains only to check that $c_1(\mathbf{k}^{*})=c_1(\mathbf{k}^{**}) $.
To see this, we first recall 
        from~\eqref{eq:choice-of-k12} that
$h|_{\mathbf{k}=\mathbf{k^{*}}}(k^{**}_{12})=0$.  
Thus, by using~\eqref{eq:h-in-proof-1-2} (and comparing with~\eqref{eq:big-sum-in-proof-1-2-moved-to-rhs}), we obtain the following:
\begin{align} \label{eq:end-of-big-proof}
     0 ~&=~ 
        (k^{**}_{12})^{n-1}
        +
        ( k^{*}_{21} -  c^{*}_{n-1}(\mathbf{k}^{*})) 
	(k^{**}_{12})^{n-2}
        -
        ( k^{*}_{21} e_1(\Sigma^*) -c^{*}_{n-2}(\mathbf{k}^{*}))
        (k^{**}_{12})^{n-3} 
        +
        \dots \\
        ~& \quad \quad \quad 
        \pm
         ( k^{*}_{21} e_{n-2}(\Sigma^*) - c^{*}_{1}(\mathbf{k}^{*}))
         ~.
         \notag
    \end{align}
Next, consider specializing the equation~\eqref{eq:big-sum-in-proof-1-2-moved-to-rhs} at $\mathbf{k}=\mathbf{k^{**}}$.  
We claim that the resulting equation is exactly what is shown in~
\eqref{eq:end-of-big-proof}, except that $c_1(\mathbf{k}^{*}) $ (in the second line) is replaced by $c_1(\mathbf{k}^{**})$.  Indeed, this claim follows readily from two facts: (1)~the parameter~$k_{21}$ and the parameters in $\Sigma$ are the same between  $\mathbf{k^{*}}$ and $\mathbf{k^{**}}$ (recall~\eqref{eq:coefficients-unchanged}), and (2)~$c_i(\mathbf{k}^{*})=c_i(\mathbf{k}^{**}) $, for $2 \leq i \leq n-1$, which we showed above.  We therefore conclude that $c_1(\mathbf{k}^{*})=c_1(\mathbf{k}^{**}) $.
\end{proof}

\subsection{Identifiability of \texorpdfstring{$\mathcal{M}_n(2,1)$}{Mn(2,1)}}
In this subsection, we analyze mammillary models with input in a peripheral compartment and output in the central compartment (Figure~\ref{fig:2-1}).  The proof of Proposition~\ref{prop:2-1} below is similar to certain parts of the proof we gave above for Proposition~\ref{prop:1-2}.

\begin{figure}[ht]
\begin{center}
	\begin{tikzpicture}[scale=1.7]
\draw (3,0) circle (0.2);
 	\draw (3.7,-1.2) circle (0.2);
 	\draw (4.4,-0.3) circle (0.2);
 	\draw (4 ,.8) circle (0.2);
    	\node[] at (3, 0) {1};
    	\node[] at (3.7, -1.2) {2};
     \node[] at (4.3, 0.3) {$\vdots$};
    	\node[] at (4.4, -0.3) {3};
    	\node[] at (4, .8) {$n$};
     \draw[->][line width = 0.55mm][green] (3.1, -.3) -- (3.5, -1);
	 \draw[<-][line width = 0.55mm][green](3.15, -.2) -- (3.55, -0.9);
	 \draw[->][blue] (3.25, -0.05) -- (4.1, -0.3);
	 \draw[<-][blue] (3.25, 0.05) -- (4.1, -0.2);
	 \draw[->][blue] (3.2, .15) -- (3.8, .6);
	 \draw[<-][blue] (3.1, 0.2) -- (3.7, .65);
   	 \node[] at (3.1, -.8) {$k_{21}$};
   	 \node[] at (3.7, -0.7) {$k_{12}$};
   	 \node[] at (3.9, .05) {$k_{13}$};
   	 \node[] at (3.8, -0.35) {$k_{31}$};
   	 \node[] at (3.25, 0.75) {$k_{1n}$};
   	 \node[] at (3.9, 0.35) {$k_{n1}$};
 	\draw (2.25, 0) circle (0.05);	
	 \draw[-] (2.3, 0) -- (2.8, 0);	
  
	 \draw[<-] (3.45, -1.3 ) -- (3, -1.3);	
   	 \node[] at (2.85,-1.3) {in};

\draw (1.7,-1.5) rectangle (5., 1.4);
	\end{tikzpicture}
\end{center}
    \caption{Parameter identifiability for $\mathcal{M}_n(2,1)$: the (green) boldface arrows indicate generically globally identifiable parameters (in fact, $k_{12}$ is globally identifiable), and all others are SLING.}
    \label{fig:2-1}
\end{figure}
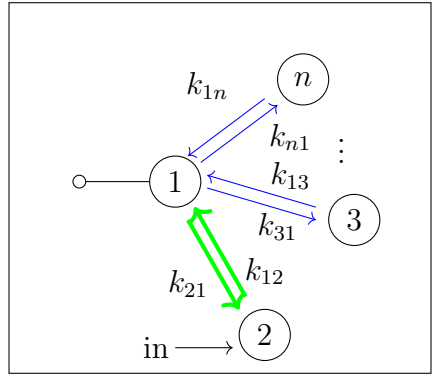

\begin{proposition}[$\mathcal{M}_n(2,1)$] \label{prop:2-1}
Let $n \geq 4$.
If $ \mathcal M =(G, \In, \Out, \Leak)$
is an $n$-compartment mammillary model 
with 
$\In=\{2\}$, 
$\Out=\{1\}$, and 
$\Leak = \varnothing$, then 
    \begin{enumerate}
        \item the parameter $k_{12}$ is globally identifiable, 
        \item the parameter $k_{21}$ is generically globally identifiable, and
        \item the parameters $k_{13}, k_{14}, \dots, k_{1n}, ~ k_{31}, k_{41}, \dots, k_{n1}$ are SLING.
    \end{enumerate}
\end{proposition}

\begin{proof} 
Part~(1) follows directly from Lemma~\ref{lem:edge-in-to-out-global} ($k_{12}$ equals the coefficient $d_{n-2}$).  Also, part~(3) follows from 
Remark~\ref{rem:1-family-uniden}, 
Lemma~\ref{lem:symmetry}, 
 and the fact that every permutation of $\{3,4,\dots,n\}$ is a model automorphism of $\mathcal{M}$ (here, the assumption $n \geq 4$ is used).

For part~(2), we begin by analyzing the coefficients on the right-hand side of the input-output equation.  According to Proposition~\ref{prop:coeff-formula}, these coefficients are sums over certain subgraphs of the following graph $G_1^*$, which is obtained from the star graph in Figure~\ref{fig:original-graph-G} by removing all the edges outgoing from  compartment-$1$ (the output compartment):

\begin{center}
	\begin{tikzpicture}[scale=1.4]
\draw (3,0) circle (0.2);
 	\draw (3.7,-1.2) circle (0.2);
 	\draw (4.4,-0.3) circle (0.2);
 	\draw (4 ,.8) circle (0.2);
    	\node[] at (3, 0) {1};
    	\node[] at (3.7, -1.2) {2};
     \node[] at (4.2, 0.4) {$\vdots$};
    	\node[] at (4.4, -0.3) {3};
    	\node[] at (4, .8) {$n$};
	 \draw[<-](3.15, -.2) -- (3.55, -0.9);
	 \draw[<-] (3.25, 0.05) -- (4.1, -0.2);
	 \draw[<-] (3.1, 0.2) -- (3.7, .65);
   	 \node[] at (3.1, -0.6) {$k_{12}$};
   	 \node[] at (3.9, .05) {$k_{13}$};
   	 \node[] at (3.25, 0.7) {$k_{1n}$};
	\end{tikzpicture}
\end{center}
In order to connect the input and output compartments (in compartment-$2$ and compartment-$1$, respectively), the subgraphs must contain the edge $k_{12}$.  The condition of being an incoming forest is vacuous (in the graph above, no compartment has more than one outgoing edge, nor can any of the edges form a cycle).  We conclude that $k_{12}$ is a factor of every right-hand side coefficient $d_i$, as in~\eqref{eq:bortner-formula-coeff}, and moreover these coefficients $d_i$ are given by the following formulas involving elementary symmetric polynomials on the set $\Sigma:=\{k_{13}, k_{14}, \dots, k_{1n} \}$:
\begin{align} \label{eq:rhs-coeff-proof-21}
    \notag
    d_{n-2} ~&=~ k_{12}\\
    \notag
    d_{n-3} ~&=~ k_{12}~ e_1(\Sigma) \\
    d_{n-4} ~&=~ k_{12}~ e_2(\Sigma) \\
    \notag
     & ~ \vdots \\
     \notag
    d_0 ~&=~ k_{12}~ e_{n-2}(\Sigma) ~=~ k_{12} k_{13} \dots k_{1n}~. 
\end{align}

Next, Proposition~\ref{prop:LHS-coefficients-updated} yields an equation~\eqref{eq:big-sum} involving left-hand side coefficients $c_i$, which we reproduce here for convenience:
\begin{align} \label{eq:big-sum-in-proof}
        \notag
        & k_{12}^{n-2} c_{n-1}
        -
        k_{12}^{n-3} c_{n-2}
        +
        k_{12}^{n-4} c_{n-3}
        -
        \dots
        \pm
         c_{1} 
        \\
        & \quad\quad  ~=~
        \left(
        k_{12}^{n-2}
        -
        k_{12}^{n-3}~ e_1(\Sigma)
        +
        k_{12}^{n-4}~ e_2(\Sigma)
        -
        \dots
        \pm
         e_{n-2}(\Sigma)
        \right) k_{21}
        ~+~
        k_{12}^{n-1}~.
    \end{align}
The idea now is to view equation~\eqref{eq:big-sum-in-proof} as a linear polynomial in $k_{21}$ with coefficients that are identifiable from the coefficient map.  More precisely, we use the equations~\eqref{eq:rhs-coeff-proof-21} to perform the following steps to obtain a nonzero linear polynomial in $k_{21}$ whose coefficients are polynomial expressions in the $c_i$'s and $d_i$'s (and thus are identifiable from the coefficient map):
    \begin{enumerate}
        \item Replace $k_{12}^{n-2}~ e_1(\Sigma)$ by 
        $k_{12}^{n-3}~d_{n-3}$, and so on, ending by replacing
        $k_{12}~ e_{n-2}(\Sigma)$ by $d_0$;
        \item Replace each occurrence of $k_{12}$ by $d_{n-2}$.
    \end{enumerate}    
Now Proposition~\ref{prop:polynom-to-iden} applies, and so $k_{21}$ is generically globally identifiable. 
\end{proof}


\subsection{Identifiability of \texorpdfstring{$\mathcal{M}_n(2,2)$}{Mn(2,2)}}
This subsection focuses on mammillary models in which the input and output are in the same peripheral compartment.  We show that the edges incident to this peripheral compartment have generically globally identifiable parameters, and all other parameters are SLING (see Figure~\ref{fig:2-2}).  We motivate this result -- and the idea behind the proof -- through the following example.

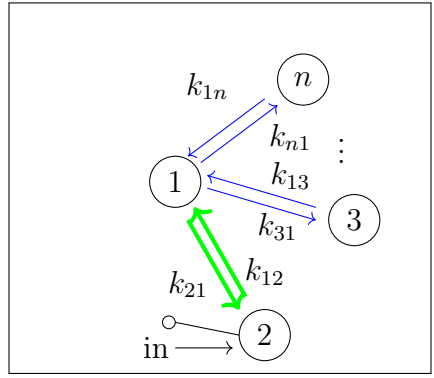
\begin{figure}[ht]
\begin{center}
	\begin{tikzpicture}[scale=1.7]
\draw (3,0) circle (0.2);
 	\draw (3.7,-1.2) circle (0.2);
 	\draw (4.4,-0.3) circle (0.2);
 	\draw (4 ,.8) circle (0.2);
    	\node[] at (3, 0) {1};
    	\node[] at (3.7, -1.2) {2};
     \node[] at (4.3, 0.3) {$\vdots$};
    	\node[] at (4.4, -0.3) {3};
    	\node[] at (4, .8) {$n$};
     \draw[->][line width = 0.55mm][green] (3.1, -.3) -- (3.5, -1);
	 \draw[<-][line width = 0.55mm][green](3.15, -.2) -- (3.55, -0.9);
	 \draw[->][blue] (3.25, -0.05) -- (4.1, -0.3);
	 \draw[<-][blue] (3.25, 0.05) -- (4.1, -0.2);
	 \draw[->][blue] (3.2, .15) -- (3.8, .6);
	 \draw[<-][blue] (3.1, 0.2) -- (3.7, .65);
   	 \node[] at (3.1, -.8) {$k_{21}$};
   	 \node[] at (3.7, -0.7) {$k_{12}$};
   	 \node[] at (3.9, .05) {$k_{13}$};
   	 \node[] at (3.8, -0.35) {$k_{31}$};
   	 \node[] at (3.25, 0.75) {$k_{1n}$};
   	 \node[] at (3.9, 0.35) {$k_{n1}$};
 	\draw (2.95, -1.1) circle (0.05);	
	 \draw[-] (3.5, -1.2) -- (3, -1.1);	
  
	 \draw[<-] (3.45, -1.3 ) -- (3, -1.3);	
   	 \node[] at (2.85,-1.3) {in};

\draw (1.7,-1.5) rectangle (5., 1.4);
	\end{tikzpicture}
\end{center}
    \caption{Parameter identifiability for $\mathcal{M}_n(2,2)$: the (green) boldface arrows indicate generically globally identifiable parameters (in fact, $k_{12}$ is globally identifiable), and all others are SLING.}
    \label{fig:2-2}
\end{figure}

\begin{example} \label{ex:M4(2,2)}
By Proposition~\ref{prop:coeff-formula}, the coefficients of the input-output equation for the model $\mathcal{M}_4(2,2)$ are as follows:
    \begin{align*}
        c_3 ~&=~k_{12} + k_{13} + k_{14} + k_{21} + k_{31} + k_{41}
            \\
        c_2 ~&=~k_{12}k_{13}+k_{12}k_{14}+ k_{12}k_{31}+k_{12}k_{41}+k_{13}k_{14}+k_{13}k_{21} +k_{13}k_{41} +k_{14}k_{21}+k_{14}k_{31}
            \\
        c_1
            ~&=~k_{12}k_{13}k_{14}+k_{12}k_{13}k_{41}+k_{12}k_{14}k_{31}+k_{13}k_{14}k_{21}
            \\
        d_2 ~&=~k_{13}+k_{14}+k_{21}+k_{31}+k_{41}
            \\
        d_1 ~&=~k_{13}k_{14}+k_{13}k_{21}+k_{13}k_{41}+k_{14}k_{21}+k_{14}k_{31}
            \\
        d_0 ~&=~ k_{13}k_{14} k_{21}
    \end{align*}
It is straightforward to check that the difference between the two linear coefficients, $c_3-d_2$, equals $k_{12}$ and so this parameter is globally identifiable.  Additionally, the difference between the degree-two coefficients, $c_2-d_1$, is a multiple of $k_{12}$ and, indeed, equals $k_{12}(d_2-k_{21})$.  As we already saw that $k_{12}$ is globally identifiable, we conclude that $k_{21}$ is generically globally identifiable (``generically'' is here because, when $k_{12}=0$, it may not be possible to recover $k_{21}$).
These ideas generalize, as seen in the next result (equation~\eqref{eq:M22-k21} in particular).
\end{example}

\begin{proposition}[$\mathcal{M}_n(2,2)$] \label{prop:2-2a}
Let $n \geq 4$.
If $ \mathcal M =(G, \In, \Out, \Leak)$
is an $n$-compartment mammillary model 
with 
$\In=\{2\}$, 
$\Out=\{2\}$, and 
$\Leak = \varnothing$, then 
    \begin{enumerate}
        \item the parameters $k_{12}$ and $k_{21}$ are generically globally identifiable (in fact, $k_{12}$ is globally identifiable), and, moreover,
        their values can be determined by the following formulas:
        \begin{align} \label{eq:M22-k21}
        \notag
            k_{12} ~&=~ c_{n-1}-d_{n-2}~, \quad {\rm and} \\
            k_{21} ~&=~ 
                d_{n-2} - \frac{c_{n-2}-d_{n-3}}{k_{12}}
                ~=~ d_{n-2} - \frac{c_{n-2}-d_{n-3}}{c_{n-1}-d_{n-2}}
            ~,
        \end{align}
        where $c_i$ and $d_i$ are   
        the coefficients of the input-output equation, as in~\eqref{eq:eq:i-o-c-and-d}--\eqref{eq:bortner-formula-coeff}; and
        \item the parameters $k_{13}, k_{14}, \dots, k_{1n}, ~ k_{31}, k_{41}, \dots, k_{n1}$ are SLING.
    \end{enumerate}
\end{proposition}

\begin{proof}
Part~(2) follows from 
Remark~\ref{rem:1-family-uniden}, 
Lemma~\ref{lem:symmetry}, 
 and the fact that every permutation of $\{3,4,\dots,n\}$ is a model automorphism of $\mathcal{M}$ (here, the assumption $n \geq 4$ is used).

For part~(1), we will use the formula for coefficents of the input-output equation (Proposition~\ref{prop:coeff-formula}).  In our model $\mathcal{M} = (G, \In, \Out, \Leak)$, the input and output are in the same compartment, so it is vacuously true 
that the input and output are always connected in spanning subgraphs of $G^*_2$ (recall that $G^*_2$ is obtained from $G$ by removing the outgoing edge $k_{12}$ from the output compartment-$2$). Therefore, the subgraph-connectedness condition for the right-hand side coefficients $d_i$, in~\eqref{eq:bortner-formula-coeff}, can be safely ignored in what follows.

We first analyze $k_{12}$.  By Proposition~\ref{prop:coeff-formula}, the linear-polynomial coefficient on the left-hand side, namely, $c_{n-1}$, is a sum over all edges of $G$. On the other hand, the corresponding coefficient on the right-hand side, $d_{n-2}$, is a sum over all edges except $k_{12}$ (because it was removed from $G$ to obtain $G^*_2$).  We conclude that $k_{12} = c_{n-1}-d_{n-2}$, as desired.

Now we consider $k_{21}$.  This time we focus on the degree-two coefficients in Proposition~\ref{prop:coeff-formula}, namely, $c_{n-2}$ and $d_{n-3}$.  Their difference is a sum over all $2$-edge spanning incoming forests of the mammillary graph such that one of the edges is $k_{12}$. The remaining edge in such a subgraph is any edge besides $k_{12}$ and $k_{21}$.  We therefore obtain the first equality here:
    \begin{align} \label{eq:difference-in-deg-2}
    \notag
        c_{n-2} - d_{n-3}
            ~&=~ 
            k_{12} (k_{13}+k_{14}+ \dots + k_{1n} ~+~ 
                    k_{31}+k_{41}+ \dots + k_{n1})
                    \\
            ~&=~ k_{12} (d_{n-2}-k_{21})~,
    \end{align}
and the second equality above follows from 
the description of $k_{21}$ given earlier in the proof.

Now it is straightforward to solve for $k_{21}$ in~\eqref{eq:difference-in-deg-2} and obtain the desired formula~\eqref{eq:M22-k21}.
\end{proof}

\subsection{Identifiability of \texorpdfstring{$\mathcal{M}_n(2,3)$}{Mn(2,3)}}
As noted earlier in Remark~\ref{rem:1-family-uniden}, the model $\mathcal{M}_n(2,3)$ contains unidentifiable parameters.  
Some parameters, however, are SLING; these parameters correspond to certain edges directed toward to the central compartment (see the [blue] solid arrows in Figure~\ref{fig:2-3}).  To illustrate the ideas behind the proof of this result (Proposition~\ref{prop:2-3} below), we present the following example.

\begin{figure}[ht]
\begin{center}
	\begin{tikzpicture}[scale=1.7]
 \draw (2.05,-.8) circle (0.2);
\draw (3,0) circle (0.2);
 	\draw (3.7,-1.2) circle (0.2);
 	\draw (4.4,-0.3) circle (0.2);
 	\draw (4 ,.8) circle (0.2);
    \node[] at (2.05,-.8){2};	
    	\node[] at (3, 0) {1};
    	\node[] at (3.7, -1.2) {3};
     \node[] at (4.3, 0.3) {$\vdots$};
    	\node[] at (4.4, -0.3) {4};
    	\node[] at (4, .8) {$n$};
     \draw[->][dashed][red] (2.8, 0) -- (2.2, -.6);
	 \draw[<-][dashed][red](2.85, -.15) -- (2.3, -.7);
       \draw[->][dashed][red] (3.1, -.3) -- (3.5, -1);
	 \draw[<-][dashed][red](3.15, -.2) -- (3.55, -0.9);
	 \draw[->][dashed][red] (3.25, -0.05) -- (4.1, -0.3);
	 \draw[<-][blue] (3.25, 0.05) -- (4.1, -0.2);
	 \draw[->][dashed][red] (3.2, .15) -- (3.8, .6);
	 \draw[<-][blue] (3.1, 0.2) -- (3.7, .65);
   \node[] at (2.3, -.2) {$k_{21}$};	
   \node[] at (2.8, -.6) {$k_{12}$};
    \node[] at (3.2, -.8) {$k_{31}$};
   	 \node[] at (3.7, -0.7) {$k_{13}$};
   	 \node[] at (3.9, .05) {$k_{14}$};
   	 \node[] at (3.8, -0.35) {$k_{41}$};
   	 \node[] at (3.25, 0.75) {$k_{1n}$};
   	 \node[] at (3.9, 0.35) {$k_{n1}$};
 	\draw (2.95, -1.1) circle (0.05);	
	 \draw[-] (3.5, -1.2) -- (3, -1.1);	
  
	 \draw[<-] (2.2, -1 ) -- (2.6, -1.3);	
   	 \node[] at (2.7,-1.3) {in};

\draw (1.7,-1.5) rectangle (5., 1.4);
	\end{tikzpicture}
\end{center}
    \caption{Parameter identifiability for $\mathcal{M}_n(2,3)$: 
    the red dashed arrows indicate parameters that we conjecture to be unidentifiable, and all others are SLING.}
    \label{fig:2-3}
\end{figure}
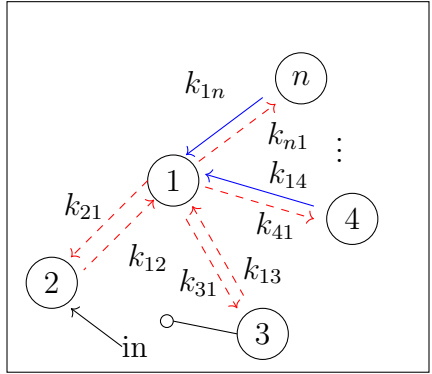

\begin{example} \label{ex:M5(2,3)}
For the model $\mathcal{M}_5(2,3)$, it is straightforward to apply Proposition~\ref{prop:coeff-formula} to obtain the following coefficients on the right-hand side of the input-output equation:
\begin{align} \label{eq:RHS-example-2-3}
        \notag
        d_3 ~&=~0
            \\
        d_2 ~&=~ k_{12}k_{31} 
            \\
        \notag
        d_1 ~&=~ k_{12}k_{31} (k_{14}+k_{15})
            \\
        \notag
        d_0 ~&=~ k_{12}k_{31} (k_{14} k_{15})~.
    \end{align}
    We see from the coefficients~\eqref{eq:RHS-example-2-3} that the following elementary symmetric polynomials are identifiable from the coefficient map: 
    $e_1(\Sigma')= k_{14}+k_{15}$ and
    $e_2(\Sigma') = k_{14}k_{15} $, where $\Sigma' =  \{k_{14}, k_{15} \}$.  This allows us to recover the set $\Sigma'$, but 
    the symmetry between compartment-$4$ and compartment-$5$ precludes us from distinguishing between $k_{14}$ and $k_{15}$.  We conclude that the parameters $k_{14}$ and $k_{15}$ are SLING. 
\end{example}

\begin{proposition}[$\mathcal{M}_n(2,3)$] \label{prop:2-3}
Let $n \geq 5$.
If $ \mathcal M =(G, \In, \Out, \Leak)$
is an $n$-compartment mammillary model 
with 
$\In=\{2\}$, 
$\Out=\{3\}$, and 
$\Leak = \varnothing$, then the parameters $k_{14}, k_{15}, \dots, k_{1n}$ are SLING.
\end{proposition}

\begin{proof}
Every permutation of $\{4,5,\dots,n\}$ is a model automorphism of $\mathcal{M}$.
Hence, by Proposition~\ref{lem:symmetry}, 
the parameters $k_{14}, k_{15}, \dots, k_{1n}$ are either all SLING or all unidentifiable
(here, the assumption $n \geq 5$ is used). Therefore, it suffices to show that $k_{14}, k_{15}, \dots, k_{1n}$ are generically locally identifiable.

Let $\Sigma':=\{k_{14}, k_{15}, \dots, k_{1n} \}$.  We claim that the right-hand side coefficients, $d_i$, as in~\eqref{eq:bortner-formula-coeff}, can be written in terms of elementary symmetric polynomials on $\Sigma'$,  as follows:
\begin{align} \label{eq:rhs-proof-2-3}
    \notag
d_{n-2} ~&=~ 0 \\
\notag
d_{n-3} ~&=~ k_{12}k_{31}
\\
    \notag
d_{n-4} ~&=~ k_{12}k_{31}~e_1(\Sigma')
\\
d_{n-5} ~&=~ k_{12}k_{31}~e_2(\Sigma')
\\
    \notag
        ~&~\vdots \\
\notag
d_{0} ~&=~ k_{12}k_{31}~e_{n-3}(\Sigma')        ~.
\end{align}
(These equations~\eqref{eq:rhs-proof-2-3} in the case of $n=5$ were shown earlier in~\eqref{eq:RHS-example-2-3}.)

We verify the formulas~\eqref{eq:rhs-proof-2-3} by using Proposition~\ref{prop:coeff-formula}, as follows.  
First, the input and output are in distinct peripheral compartments, so they can not be connected by a single edge.  Thus, $d_{n-2} = 0$.  Next, 
the only spanning incoming forests with two edges that connect the input to the output here are $k_{12}k_{31}$ and $k_{21}k_{13}$, and since the edge $k_{13}$ is deleted in $G^*$, we must have $d_{n-3}=k_{12}k_{31}$. 
Similarly, all terms in the remaining coefficients $d_i$, for $i=n-2,n-1,\dots, 0$, must have a factor of $k_{12}k_{31}$, since this is the only pair of edges that can take part in a spanning incoming forest that connects the input and output. If we begin with $k_{12}$ and $k_{31}$ and then try to add more edges, the only edges we can add without creating cycles or having more than one outgoing edge from compartment-$1$ are the edges 
of $\Sigma' = \{k_{14}, k_{15}, \dots, k_{1n}\}$.  Indeed, such edges can be added arbitrarily, which yields the elementary symmetric polynomials shown in~\eqref{eq:rhs-proof-2-3}. 

Next, the equations~\eqref{eq:rhs-proof-2-3} imply that the elementary symmetric polynomials appearing in these equations, namely, 
$e_1(\Sigma'), e_2(\Sigma'), \dots, e_{n-3}(\Sigma')$, are identifiable from the coefficient map.  Now Lemma~\ref{lem:deg-elem-sym} implies that each parameter $k_{14}, k_{15}, \dots, k_{1n}$, when viewed as a (projection) function $\mathbb{R}^{|E| + |\Leak|} \to \mathbb{R}$, is locally identifiable from the coefficient map.  Hence, by Remark~\ref{rem:iden-fnc-vs-param}, 
each of the parameters $k_{14}, k_{15}, \dots, k_{1n}$
is generically locally identifiable.
\end{proof}

\begin{remark}[$\mathcal{M}_4(2,3)$] \label{rem:formula}
The ideas in the proof of Proposition~\ref{prop:2-3} can be used to show that, in the $n=4$ version of the model $\mathcal{M}_n(2,3)$, 
the parameter $k_{14}$ is generically globally identifiable. Indeed, in this case, the equations~\eqref{eq:rhs-proof-2-3} imply that $k_{14}= \frac{d_0}{d_1}$.    
\end{remark}

Next, we conjecture that all parameters not considered in Proposition~\ref{prop:2-3} are unidentifiable, as follows (cf.\ Figure~\ref{fig:2-3}).

\begin{conjecture}[$\mathcal{M}_n(2,3)$] \label{conj:2-3}
Let $n \geq 5$.
If $ \mathcal M =(G, \In, \Out, \Leak)$
is an $n$-compartment mammillary model 
with 
$\In=\{2\}$, 
$\Out=\{3\}$, and 
$\Leak = \varnothing$, then the following parameters are unidentifiable:
\[
k_{12}~, \quad
k_{21}~, \quad 
k_{13}~, \quad {\rm and}
    \quad 
k_{31}~, k_{41}, \dots, k_{n1}~.
\]
\end{conjecture}

\begin{remark}  \label{rem:2-3}
    By Lemma~\ref{lem:symmetry}, 
    the parameters 
    $k_{41}, k_{51}, \dots, k_{n1}$ in 
the model $\mathcal{M}_n(2,3)$ all have the same type of identifiability.  Also, the equation $d_{n-3}=k_{12} k_{31}$,  from equation~\eqref{eq:rhs-proof-2-3} in the proof of Proposition~\ref{prop:2-3}, implies that $k_{12}$ and $k_{31}$ have the same type of identifiability.   Conjecture~\ref{conj:2-3} asserts that all of these parameters are, in fact, unidentifiable.
\end{remark}

\section{Discussion} \label{sec:discussion}
In this work, we prove the identifiability properties of individual parameters (specifically, generically globally identifiable versus SLING) in mammillary models with one input, one output, and no leaks (Theorem~\ref{thm:summary}). As noted earlier, we are essentially the first to undertake such an investigation for an infinite family of models.  We therefore expect that our work to lead the way for future studies, involving other families of models.  

A natural candidate for such future investigations is the family of {\em catenary models}, those in which the underlying graph is a bidirected path (see~\cite{cobelli-lepschy-romaninjacur} and~\cite[\S 5.8]{cobelli2002identifiability} for partial results in this direction).  A starting point for such a project is the database of catenary models with three compartments that appears in the recent work of Dessauer {\em et al.}~\cite[Appendix~B]{cycle-cat}.  
In this database, each model is shown together with the identifiability properties of all parameters.  
We also direct the reader to additional databases of linear compartmental models, one for models with three compartments (created by Norton)~\cite{norton1982investigation} and one for models with up to four compartments (created by Gogishvili)~\cite{gogishvili-database}.

Returning our attention to mammillary models, there are a number of future directions.  The first is to resolve Conjecture~\ref{conj:2-3}, which pertains to the parameters in model~$\mathcal{M}_n(2,3)$ that we believe are unidentifiable (as shown in Figure~\ref{fig:summary}).  Additionally, it is natural to ask how our results extend to allow for (mammillary models with) any number of inputs, outputs, and leaks.  In other words, 
how does Theorem~\ref{thm:summary} generalize when inputs, output, and/or leaks are added?  This question is open, although some partial results were proven by Chau~\cite{chau-mam} and Cobelli, Lepschy, and Romanin Jacur~\cite{cobelli-lepschy-romaninjacur} (see also~\cite[\S 5.8]{cobelli2002identifiability}).

A version of this question -- at the level of models rather than their parameters -- has been addressed in general~\cite{CJSSS,GOS, gogishvili-database,GHMS}, but we are not aware of results in this direction pertaining to individual parameters. 
The aforementioned questions -- which investigate the effect on identifiability of various operations on models -- are important in the context of systems biology, because a model might not be completely known or we may want to predict the effect of various interventions that lead to changes in the model.  For instance, we might not know precisely which compartments have leaks, and so we would want to answer the question of identifiability for a family of models arising from various possible sets of leaks.  Additionally, although we considered models without leaks in this work, we recognize that models that do contain leaks are more realistic in biological settings, as there are very few systems that are completely closed. Finally, allowing for more inputs or outputs represents additional controls or measurements of the system, which again may better reflect the situation in experimental settings.  In short, our results take the first step into an investigation with direct relevance for biological applications.


{\small
\subsection*{Acknowledgements}
This research was initiated by Katherine Clemens, Jonathan Martinez, and Michaela Thompson at the 2023 REU in the Department of Mathematics at Texas A\&M University, in which Anne Shiu and Benjamin Warren were research mentors. 
The REU was supported by the NSF (DMS-2150094); and 
AS and BW were partially supported by the NSF (DMS-1752672).  The authors thank Cashous Bortner and Alejandro F. Villaverde for helpful feedback on an earlier version of this work.  The authors also acknowledge an anonymous reviewer whose insightful comments improved this work.

\subsection*{Competing interests}
The authors have no competing interests to declare that are relevant to the content of this article.

\subsection*{Data availability statement}
This work does not have any associated data.
}

\bibliographystyle{plain}
\bibliography{references}

\end{document}